\documentclass[a4paper, 10pt]{amsart}
\usepackage{amsmath, amssymb, color, hyperref}

\theoremstyle{plain}
\newtheorem{theorem}{\bf Theorem}[section]
\newtheorem{proposition}[theorem]{\bf Proposition}
\newtheorem{lemma}[theorem]{\bf Lemma}
\newtheorem{corollary}[theorem]{\bf Corollary}

\theoremstyle{definition}
\newtheorem{example}[theorem]{\bf Example}

\newtheorem{definition}[theorem]{\bf Definition}
\newtheorem{remark}[theorem]{\bf Remark}

\newcommand{\N}{\mathbb N}
\newcommand{\Z}{\mathbb Z}
\newcommand{\R}{\mathbb R}
\newcommand{\Q}{\mathbb Q}

\newcommand{\BF}{\text{\rm BF}}
\newcommand{\FF}{\text{\rm FF}}
\newcommand{\C}{\text{\rm C}}
\newcommand{\G}{\text{\rm G}}

\DeclareMathOperator{\spec}{spec} \DeclareMathOperator{\supp}{supp}

\DeclareMathOperator{\id}{id}

\newcommand{\DP}{\negthinspace :\negthinspace}
\newcommand{\red}{\text{\rm red}}
\newcommand{\eq}{\text{\rm eq}}
\newcommand{\adj}{\text{\rm adj}}
\newcommand{\mon}{\text{\rm mon}}

\renewcommand{\time}{\negthinspace \times \negthinspace}
\renewcommand{\t}{\, | \,}

\numberwithin{equation}{section}

\begin{document}

\title[The monotone catenary degree]{The monotone catenary degree of monoids of ideals}

\address{Institut f\"ur Mathematik und wissenschaftliches Rechnen, Karl-Franzens-Universit\"at Graz, NAWI Graz, Heinrichstra{\ss}e 36, 8010 Graz, Austria}
\email{alfred.geroldinger@uni-graz.at, andreas.reinhart@uni-graz.at}
\author{Alfred Geroldinger and Andreas Reinhart}
\urladdr{http://imsc.uni-graz.at/geroldinger, http://imsc.uni-graz.at/reinhart}
\thanks{This work was supported by the Austrian Science Fund FWF, Project Number P26036-N26.
\\}

\begin{abstract}
Factoring ideals in integral domains is a central topic in multiplicative ideal theory. In the present paper we study monoids of ideals and consider factorizations of ideals into multiplicatively irreducible ideals. The focus is on the monoid of nonzero divisorial ideals and on the monoid of $v$-invertible divisorial ideals in weakly Krull Mori domains. Under suitable algebraic finiteness conditions we establish arithmetical finiteness results, in particular for the monotone catenary degree and for the structure of sets of lengths and of their unions.
\end{abstract}

\maketitle

\medskip
\section{Introduction}\label{1}
\medskip

Factoring ideals in integral domains is a central topic in multiplicative ideal theory (for a monograph reflecting recent developments we refer to \cite{Fo-Ho-Lu13a}).
In the present paper we study monoids of ideals, consider factorizations of ideals into multiplicatively irreducible ideals, and prove finiteness results on the monotone catenary degree and structural results on sets of lengths. First we recall the concept of the monotone catenary degree and then we discuss the monoids of ideals under consideration.

Let $H$ be an atomic monoid which is not factorial. Then every non-unit can be written as a finite product of atoms and there is an element $a\in H$ having at least two distinct factorizations, say $a = u_1\cdot\ldots\cdot u_{\ell} v_1\cdot\ldots\cdot v_m = u_1\cdot\ldots\cdot u_{\ell} w_1\cdot\ldots\cdot w_n$ where all $u_i, v_j, w_k$ are atoms and the $v_j$ and $w_k$ are pairwise not associated. Then $\ell+m$ and $\ell+n$ are the lengths of the two factorizations and $\max \{m,n\} > 0$ is their distance. Then for every $M\in \N$, the element
\[
a^M = (u_1\cdot\ldots\cdot u_{\ell} v_1\cdot\ldots\cdot v_m)^{\nu} (u_1\cdot\ldots\cdot u_{\ell} w_1\cdot\ldots
\cdot w_n)^{M- \nu} \quad \text{for all } \quad \nu\in [0,M]
\]
has factorizations with distance greater than $M$. However, at least those factorizations, which are powers of factorizations of $a$, can be concatenated, step by step, by factorizations whose distance is small and does not depend on $M$. This phenomenon is formalized by the catenary degree which is defined as follows. The catenary degree $\mathsf c (H)$ of $H$ is the smallest $N\in \N_0 \cup \{\infty\}$ such that for each $a\in H$ and each two factorizations $z, z'$ of $a$ there is a concatenating chain of factorizations $z = z_0, z_1, \ldots, z_{k+1}=z'$ of $a$ such that the distance $\mathsf d (z_{i-1}, z_i)$ between two successive factorizations is bounded by $N$. It is well-known that the catenary degree is finite for Krull monoids with finite class group and for C-monoids (these include Mori domains $R$ with nonzero conductor $\mathfrak f = (R \colon \widehat R)$ for which the residue class ring $\widehat R/\mathfrak f$ and the class group $\mathcal C ( \widehat R)$ are finite).

In order to study further structural properties of concatenating chains, Foroutan introduced the monotone catenary degree (\cite{Fo06a}). The monotone catenary degree $\mathsf c_{\mon} (H)$ is the smallest $N\in \N_0 \cup \{\infty\}$ such that for each $a\in H$ and each two factorizations $z, z'$ of $a$ there is a concatenating chain of factorizations $z = z_0, z_1, \ldots, z_{k+1}=z'$ of $a$ such that the distance between two successive factorizations is bounded by $N$ and in addition the sequence of lengths $|z_0|, \ldots, |z_{k+1}|$ of the factorizations is monotone (thus either $|z_0|\le \ldots \le |z_{k+1}|$ or $|z_0|\ge \ldots \ge |z_{k+1}|$). Therefore, by definition, we have $\mathsf c (H) \le \mathsf c_{\mon} (H)$ and Foroutan showed that the monotone catenary degree of Krull monoids with finite class group is finite again. Subsequently the monotone catenary degree, or more generally, monotonic properties of concatenating chains were studied in a variety of papers (e.g., \cite{Bl-Ga-Ge11a,Fo-Ha06b,Ge-Gr-Sc-Sc10,Ge-Yu13a,Ha09c,Ph12b}). We mention one result in detail, namely that C-monoids have the following property (\cite[Theorem 1.1]{Fo-Ge05}):

\begin{itemize}
\item[] There exists some constant $M\in \N$ such that for every $a\in H$ and for each two factorizations $z, z'$ of $a$ there exist factorizations $z=z_0, z_1, \ldots, z_{k+1} = z'$ of $a$ such that, for every $i\in [1, k+1]$,
\[
\mathsf d (z_{i-1}, z_i ) \le M \quad \text{ and \ $($either} \
|z_1| \le \ldots \le |z_k| \ \text{or} \ |z_1| \ge \ldots \ge
|z_k| \ ).
\]
\end{itemize}
Thus the sequence of lengths of factorizations is monotone apart from the first and the last step. However, there are C-monoids having infinite monotone catenary degree and this may happen even for one-dimensional local Noetherian domains (see Remark \ref{5.3}).

In Section \ref{4} we study integral domains $R$ with complete integral closure $\widehat R$, nonzero conductor $(R \colon \widehat R)$, and with an $r$-Noetherian ideal system $r$ such that $r$-$\max(R)= \mathfrak X(R)$. Under suitable algebraic finiteness conditions the monoid of nonzero $r$-ideals $\mathcal I_r(R)$ is finitely generated up to a free abelian factor. However, since the monoid need not be cancellative, its study needs some semigroup theoretical preparations (done in Section \ref{3}) valid for unit-cancellative  monoids. The main arithmetical result on $\mathcal I_r(R)$ is given in Theorem \ref{4.15}, and for its assumptions see Theorems \ref{4.8} and \ref{4.13}.

In Section \ref{5} we study the monoid $\mathcal I_v^*(H)$ of $v$-invertible $v$-ideals of $v$-Noetherian weakly Krull monoids with nontrivial conductor. This is a C-monoid and isomorphic to a finite direct product of finitely primary monoids and free abelian factor. However, in general, the finiteness of the monotone catenary degree is not preserved when passing to finite direct products (Example \ref{3.3}). In order to overcome this difficulty we introduce a new arithmetical invariant, the weak successive distance (Definition \ref{3.4}), and establish a result allowing to shift the finiteness of the monotone catenary degree to finite direct products (Theorem \ref{3.8}). This is done under the additional assumption that the Structure Theorem for Sets of Lengths holds (which is the case for $\mathcal I_v^*(H)$). Along our way we prove that the Structure Theorem for Unions holds, both for the monoid $\mathcal I_r(R)$ (studied in Section \ref{4}) as well as for $\mathcal I_v^*(H)$. This is based on a recent characterization result for the validity of the Structure Theorem for Unions obtained in \cite{F-G-K-T17}. The main results on $\mathcal I_v^*(H)$ are given in Theorem \ref{5.13} and Corollary \ref{5.14}.

\medskip
\section{Background on the arithmetic of monoids} \label{2}
\medskip

We denote by $\N$ the set of positive integers, and we put $\N_0 =
\N\cup\{0\}$. For $s\in \N$, we will consider the product order on $\mathbb N_0^s$ which is induced by the usual order of $\mathbb N_0$.
For real numbers $a,b\in\R$, we
set $[a,b] =\{ x\in\Z\mid a\le x\le b\}$. Let $L,L'\subset
\Z$. We denote by $L+L' =\{a+b\mid a\in L,\,b\in L'\}$ their
{\it sumset}. Two distinct elements $k, \ell\in L$ are called {\it
adjacent} if $L\cap [\min\{k,\ell\},\max\{k,\ell\}]=\{k,\ell\}$. A positive
integer $d\in\N$ is called a\ {\it distance}\ of $L$\ if there
exist adjacent elements $k,\ell\in L$ with $d=|\ell - k|$, and we denote by $\Delta (L)$ the {\it set of distances} of $L$. If $L \subset \N$, we denote by $\rho (L) = \sup L/\min L\in \Q_{\ge 1} \cup \{\infty\}$ the {\it elasticity} of $L$. We set $\rho (\{0\})=1$ and $\max\emptyset=\min\emptyset=\sup\emptyset=0$.

\smallskip
\noindent {\bf Monoids and factorizations}. All rings and semigroups are commutative and have an identity element. Let $H$ be a semigroup. If not stated otherwise, we use multiplicative notation and $1=1_H\in H$ means the identity element of $H$. We denote by $H^{\times}$ the group of invertible elements of $H$ and we say that $H$ is reduced if $H^{\times}=\{1\}$. Then $H_{\red} = \{aH^{\times} \mid a\in H \}$ denotes the associated reduced semigroup of $H$. Furthermore, $H$ is called
\begin{itemize}
\item {\it unit-cancellative} if $a,u\in H \quad \text{and} \quad a= au, \quad \text{implies that} \quad u\in H^{\times}$,

\item {\it cancellative} if $a,b,u\in H$ and $au = bu$ implies that $a=b$.
\end{itemize}
Obviously, every cancellative semigroup is unit-cancellative. The property of being unit-cancellative is a frequently studied property both for rings and semigroups. Indeed, a commutative ring $R$ is called pr\'esimplifiable if $a,u\in R$ and $a = au$ implies that either $a=0$ or $u\in R^{\times}$.
This concept was introduced by Bouvier and further studied by D.D. Anderson et al. (\cite{An-Al17a, An-VL96a, Bo72a, Fr-Fr11a}).

If $H$ is cancellative, then $\mathsf q (H)$ denotes the
quotient group of $H$ and
\[
\widehat H = \{ x\in \mathsf q (H) \mid \text{there is a $c\in H$ such that $cx^n\in H$ for every $n\in \N$} \} \subset \mathsf q (H)
\]
is the {\it complete integral closure} of $H$. Let $R$ be a domain with quotient field $K$. For every subset $X \subset K$ we set $X^{\bullet} = X \setminus \{0\}$. Then $R^{\bullet}$ is a cancellative semigroup, $\overline R$ is the integral closure of $R$, and $\widehat R = \widehat{R^{\bullet}} \cup \{0\}$ is the complete integral closure of $R$.

\medskip
\centerline{\it Throughout this paper, a monoid will always mean a}
\centerline{\it commutative unit-cancellative semigroup with identity element.}
\medskip

For a set $\mathcal P$, we denote by $\mathcal F (\mathcal P)$ the\ {\it free
abelian monoid}\ with basis $\mathcal P$. Then every $a\in\mathcal F
(\mathcal P)$ has a unique representation in the form
\[
a =\prod_{p\in \mathcal P} p^{\mathsf v_p(a) }\quad\text{with}\quad
\mathsf v_p(a)\in\N_0\,\text{ and }\,\mathsf v_p(a) = 0\quad\text{
for almost all }\ p\in \mathcal P\,.
\]
We call $|a|=\sum_{p\in \mathcal P}\mathsf v_p(a)$ the \emph{length} of $a$
and $\supp (a) =\{p\in \mathcal P\mid\mathsf v_p (a) > 0\}\subset \mathcal P$ the
{\it support} of $a$.

Let $H$ be a monoid. A submonoid $S \subset H$ is said to be
\begin{itemize}
\item {\it divisor-closed} if $a\in S$ and $b\in H$ with $b \t a$ implies that $b\in S$,

\item {\it saturated} if $a, c\in S$ and $b\in H$ with $a=bc$ implies that $b\in S$.
\end{itemize}
Clearly, every divisor-closed submonoid is saturated.
An element $u\in H$ is said to be {\it irreducible} (or an {\it atom}) if $u \notin H^{\times}$ and an equation $u=ab$ with $a, b\in H$ implies that $a\in H^{\times}$ or $b\in H^{\times}$.
Then $\mathcal A (H)$ denotes the set of atoms of $H$, and $H$ is said to be {\it atomic} if every non-unit can be written as a finite product of atoms of $H$. A simple argument shows that $H$ is atomic whenever the ACCP (ascending chain condition on principal ideals) holds (\cite[Lemma 3.1]{F-G-K-T17}). From now on we suppose that $H$ is atomic.

The free abelian monoid $\mathsf Z (H) =\mathcal F\bigl(
\mathcal A(H_\red)\bigr)$\ is called the\ {\it factorization
monoid}\ of $H$, and the homomorphism
\[
\pi\colon\mathsf Z (H)\to H_{\red}\quad\text{satisfying}\quad
\pi (u) = u\quad\text{for each}\quad u\in\mathcal A(H_\red)
\]
is called the\ {\it factorization homomorphism}\ of $H$. For $a
\in H$ and $k\in\N$,
\[
\begin{aligned}
\mathsf Z_H (a) = \mathsf Z (a) & =\pi^{-1} (aH^\times)\subset
\mathsf Z (H)\quad
\text{is the\ {\it set of factorizations}\ of\ $a$}\,,\\
\mathsf Z_{H,k} (a) = \mathsf Z_k (a) & =\{ z\in\mathsf Z (a)\mid |z| = k\}\quad
\text{is the\ {\it set of factorizations}\ of\ $a$ of length\
$k$},\quad\text{and}
\\
\mathsf L_H (a) = \mathsf L (a) & =\bigl\{ |z|\,\bigm|\, z\in
\mathsf Z (a)\bigr\}\subset\N_0\quad\text{is the\ {\it set of
lengths}\ of $a$}\,.
\end{aligned}
\]
If $S \subset H$ is a divisor-closed submonoid and $a\in S$, then $\mathsf Z_S (a) = \mathsf Z_H (a)$ whence $\mathsf L_S (a) = \mathsf L_H (a)$.
Then
\begin{itemize}
\item $\mathcal L (H) = \{\mathsf L (a) \mid a\in H \}$ is the {\it system of sets of lengths} of $H$,

\item $\Delta (H) =\bigcup_{L\in \mathcal L (H)} \Delta \big(L \big)$ is the {\it set of distances} of $H$, and

\item $\rho (H) = \sup \{ \rho (L) \mid L\in \mathcal L (H) \}$ is the {\it elasticity} of $H$.
\end{itemize}
The monoid $H$ is said to be
\begin{itemize}
\item {\it half-factorial} if $\Delta (H) =\emptyset$ (equivalently, $|\mathsf L (a)|=1$ for all $a\in H$),

\item an {\it \FF-monoid} if $\mathsf Z (a)$ is finite for all $a\in H$,

\item a {\it \BF-monoid} if $\mathsf L (a)$ is finite for all $a\in H$.
\end{itemize}

\smallskip
Let $z,\,z'\in\mathsf Z (H)$. Then we can write
\[
z = u_1\cdot\ldots\cdot u_{\ell} v_1\cdot\ldots\cdot v_m\quad
\text{and}\quad z' = u_1\cdot\ldots\cdot u_{\ell} w_1\cdot\ldots
\cdot w_n\,,
\]
where $\ell,\,m,\,n\in\N_0$ and $u_1,\ldots,u_\ell,\,v_1,
\ldots,v_m,\,w_1,\ldots,w_n\in\mathcal A(H_\red)$ are such that
\[
\{v_1,\ldots,v_m\}\cap\{w_1,\ldots,w_n\} =\emptyset\,.
\]
Then $\gcd(z,z')=u_1\cdot\ldots\cdot u_{\ell}$, and we call
\[
\mathsf d (z,z') =\max\{m,\,n\} =\max\{ |z\gcd (z,z')^{-1}|,
|z'\gcd (z,z')^{-1}|\}\in\N_0
\]
the {\it distance} between $z$ and $z'$. If $\pi (z) =\pi (z')$ and
$z\ne z'$, then
\begin{equation}\label{E:Dist}
1 + \bigl| |z |-|z'|\bigr|\le\mathsf d (z,z') \quad \text{resp.} \quad 2 + \bigl| |z |-|z'|\bigr|\le\mathsf d (z,z') \quad \text{if $H$ is cancellative}
\end{equation}
(see \cite[Proposition 3.2]{F-G-K-T17} and \cite[Lemma 1.6.2]{Ge-HK06a}). For subsets $X,Y\subset\mathsf Z(H)$, we set
\[
\mathsf d(X,Y) =\min\{\mathsf d(x,y)\mid x\in X,\,y\in
Y\}\textnormal{ and}
\]
\[
{\rm Dist}(X,Y)=\sup\{\mathsf d (\{x\},Y),\mathsf d (X,\{y\})\mid x\in X,\,y\in
Y\}.
\]

\smallskip
Note that $\mathsf d (X,Y) = 0$ if and only if ( $X\cap Y\ne
\emptyset$ or $X=\emptyset$ or $Y =\emptyset$ ). For a factorization $z\in \mathsf Z(H)$, we denote by \ $\delta (z)$ \ the smallest $N\in \N_0$ with the following property:

\begin{enumerate}
\item[]
If $k\in \N$ is such that $k$ and $|z|$ are adjacent lengths of $\pi(z)$, then there exists some $y\in \mathsf Z(H)$ such that $\pi(y) = \pi(z)$, \ $|y| = k$ and \ $\mathsf d(z,y) \le N$.
\end{enumerate}

\noindent
We call
\[
\delta (H) = \sup \{\, \delta (z) \mid z\in Z(H) \}\in \N_0 \cup
\{\infty\}
\]
the \ {\it (strong) successive distance} \ of $H$. Note that
\[
\delta (H) = \sup \big\{ {\rm Dist}(\mathsf Z_k (a), \mathsf Z_{\ell} (a) ) \mid a \in H, k,\ell \in \mathsf L (a) \ \text{are adjacent}    \big\} \,.
\]

\medskip
\noindent
{\bf Chains of factorizations}. Let $a\in H$ and
$N\in\mathbb N_0\cup\{\infty\}$. A finite sequence $z_0,\ldots,
z_k\in\mathsf Z (a)$ is called a {\it $($monotone$)$ $N$-chain of
factorizations of $a$} if $\mathsf d (z_{i-1},z_i)\le N$ for all $i\in
[1,k]$ (and $|z_0|\le\ldots\le |z_k|$ or $|z_0|\ge\ldots\ge
|z_k|$). We denote by $\mathsf c (a)$ (or by $\mathsf c_{\mon} (a)$ resp.) the smallest $N\in\N _0\cup\{\infty\}$ such
that any two factorizations $z,\,z'\in\mathsf Z (a)$ can be
concatenated by an $N$-chain (or by a monotone $N$-chain resp.).
Then
\[
\mathsf c(H) =\sup\{\mathsf c(b)\mid b\in H\}\in\N_0\cup
\{\infty\}\quad\text{and}\quad\mathsf c_{\mon} (H) =\sup\{
\mathsf c_{\mon} (b)\mid b\in H\}\in\N_0\cup\{\infty\}\quad
\,
\]
denote the\ {\it catenary degree}\ and the\ {\it monotone
catenary degree} of $H$. The monotone catenary degree is studied by
using the two auxiliary notions of the equal and the adjacent
catenary degrees. Let $\mathsf c_{\eq} (a)$ denote the smallest $N
\in\mathbb N_0\cup\{\infty\}$ such that any two factorizations $z,z'\in\mathsf Z (a)$ with $|z| = |z'|$ can be concatenated by a monotone $N$-chain.
We call
\[
\mathsf c_{\eq} (H) =\sup\{\mathsf c_{\eq} (b)\mid b\in H\}\in\mathbb N_0\cup\{\infty\}
\]
the {\it equal catenary degree} of $H$.
We set
\[
\mathsf c_{\adj}(a) =\sup\{\mathsf d\big(\mathsf Z_k (a),
\mathsf Z_{\ell} (a)\big)\mid k, \ell\in\mathsf L (a)\,\text{are
adjacent}\}\,,
\]
and the {\it adjacent catenary degree} of $H$ is defined as
\[
\mathsf c_{\adj} (H) =\sup\{\mathsf c_{\adj} (b)\mid b\in H\}\in\mathbb N_0\cup
\{\infty\}\,.
\]
Obviously, we have $\mathsf c_{\adj} (H) \le \delta (H)$,
\[
\mathsf c (a)\le\mathsf c_{\mon} (a) =\sup\{\mathsf c_{\eq}
(a),\mathsf c_{\adj} (a)\}\le\sup\mathsf L (a)\quad\text{for
all}\quad a\in H\,,
\]
and hence
\begin{equation}
\mathsf c (H)\le\mathsf c_{\mon} (H) =\sup\{\mathsf c_{\eq}
(H),\mathsf c_{\adj} (H)\}\,.\label{basic2}
\end{equation}
Note that $\mathsf c_{\adj} (H) = 0$ if and only if $H$ is half-factorial and if this holds, then $\mathsf c_{\eq} (H)=\mathsf c (H)$. If $H$ is not half-factorial, then (\ref{E:Dist}) shows that $1+\sup\Delta (H)\le\mathsf c (H)$. Moreover, $\mathsf c_{\eq}(H) = 0$ if and only if for all $a\in H$ and all $k\in\mathsf L (a)$ we have $|\mathsf Z_k (a)| = 1$. A result by Coykendall and Smith implies that
for the multiplicative monoid $H$ of non-zero elements from a domain we have $\mathsf c_{\eq}(H) = 0$ if and only if
$H$ is factorial (\cite[Corollary 2.12]{Co-Sm11a}).

\medskip
\noindent
{\bf Structure of sets of lengths}. Let $d\in\N$,\ $M\in\N_0$\ and\ $\{0,d\}\subset\mathcal D
\subset [0,d]$. Then $L$ is called an {\it almost
arithmetical multiprogression}\ ({\rm AAMP}\ for
short)\ with\ {\it difference}\ $d$,\ {\it period}\ $\mathcal D$,
\ and\ {\it bound}\ $M$,\ if
\[
L = y + (L'\cup L^*\cup L'')\,\subset\,y +\mathcal D + d\Z
\]
where
\begin{itemize}
\item $L^*$ is finite and nonempty with $\min L^* = 0$ and $L^* =(\mathcal D + d\Z)\cap [0,\max L^*]$
\item $L'\subset [-M,-1]$\ and\ $L''\subset\max L^* + [1,M]$
\item $y\in\Z$.
\end{itemize}
Note that an AAMP is finite and nonempty. It is straightforward to prove that if $M\in \N_0, \ d\in\mathbb{N}$, $L$ is an AAMP with bound $M$ and difference $d$, $x\in L$ and $y\in d\mathbb{Z}$ are such that $\min L+M\leq x+y\leq\max L-M$, then $x+y\in L$.

We say that the\ {\it Structure Theorem for Sets of Lengths} holds (for the monoid $H$)\ if $H$ is atomic and there exist some $M\in\N_0$ and a finite nonempty set $\Delta \subset \N$ such that for every $a\in H$, the set of lengths $\mathsf L (a)$ is an {\rm AAMP} with some difference $d\in\Delta $ and bound $M$.
Suppose that the Structure Theorem for Sets of Lengths holds for the monoid $H$. Then $H$ is a \BF-monoid with finite set of distances. The monoids of ideals (studied in Section \ref{5}) satisfy the Structure Theorem for Sets of Lengths and we use this property to show that their monotone catenary degree is finite (see Theorem \ref{3.8}).

\medskip
\noindent
{\bf Unions of sets of lengths.} For every $k\in \N$,
\[
\mathcal U_k (H) = \bigcup_{L\in \mathcal L (H) \ \text{with} \ k\in L } L
\]
denotes the {\it union of all sets of lengths containing $k$}, provided that $H \ne H^{\times}$. In the extremal case where $H = H^{\times}$, it is convenient to set $\mathcal U_k (H) = \{k\}$ for all $k\in \N$. Furthermore,
\[
\rho_k (H) = \sup \mathcal U_k (H)\in \N_0 \cup \{\infty\}
\]
denotes the $k$-th {\it elasticity} of $H$. Unions of sets of lengths are a classic invariant in factorization theory, and the last decade has seen a renewed interest in the structure of unions (e.g., \cite{F-G-K-T17, Ga-Ge09b, Tr18a}).

We say that the \ {\it Structure Theorem for Unions} \ holds (for the monoid $H$) if there are $d\in \N$ and $M\in \N_0$ such that, for all sufficiently large $k\in \N$, $\mathcal U_k (H)$ has the form
\[
\mathcal U_k (H) = y_k + (L_k' \cup L_k^* \cup L_k'') \subset y_k + d \Z
\]
where $L_k^*$ is a nonempty arithmetical progression with difference $d$ such that $\min L_k^*=0$, $L_k' \subset [-M,-1]$, $L_k'' \subset \sup L_k^* + [1,M]$ (with the convention that $L_k''=\emptyset$ if $L_k^*$ is infinite), and $y_k\in \Z$. Note, if $\mathcal U_k (H)$ is finite, then $\mathcal U_k (H)$ is an AAMP with period $\{1,d\}$ and bound $M$.

Suppose that $\Delta (H)$ is finite and the Structure Theorem for Unions holds. Then, by \cite[Corollary 2.3 and Lemma 2.12]{F-G-K-T17}, the parameter $d$ in the above definition satisfies $d = \min \Delta (H)$ and we have
\[
\lim_{k \to\infty}\frac{|\mathcal U_k (H)|}{k} = \frac{1}{d} \Big( \rho (H) - \frac{1}{\rho (H)} \Big) \,.
\]

\medskip
\section{Finitely generated monoids and finite direct products} \label{3}
\medskip

Let $H$ be a monoid and $\pi \colon \mathsf Z (H) \to H_{\red}$ the canonical epimorphism. Its monoid of relations, defined as
\[
\sim_H \ = \{ (x,y)\in \mathsf Z (H) \times \mathsf Z (H) \mid \pi (x)= \pi (y) \} \,,
\]
is a crucial tool for studying the arithmetic of $H$. Suppose that $H$ is reduced, cancellative, and atomic. Then $\sim_H \ \subset \mathsf Z (H) \times \mathsf Z (H)$ is a saturated submonoid and hence a Krull monoid (\cite[Lemma 11]{Ph10a}). If, moreover, $H$ is finitely generated, then $\sim_H$ is finitely generated. However, this need not be true without the assumption that $H$ is cancellative (\cite[Remarks 3.11]{F-G-K-T17}) (although, by Redei's Theorem, $\ker ( \pi )$ is finitely generated as a congruence and $H$ is finitely presented). A further striking difference between the cancellative case and the non-cancellative case is, that cancellative finitely generated monoids have accepted elasticity which need not be the case in the noncancellative setting (for an example we refer again to \cite[Remarks 3.11]{F-G-K-T17}).

In spite of all these differences, we show in our next result that finitely generated monoids have finite successive distance and finite monotone catenary degree (both results are known in the cancellative setting, \cite[Theorem 3.1.4]{Ge-HK06a} and \cite[Theorem 3.9]{Fo06a}).

\medskip
\begin{theorem} \label{3.1}
Let $H$ be a monoid such that $H_{\red}$ is finitely generated. Then $\delta (H) <\infty$ and $\mathsf c_{\mon} (H) <\infty$.
\end{theorem}

\begin{proof}
By \cite[Proposition 3.4]{F-G-K-T17}, $H$ is a BF-monoid with finite set of distances. Without restriction we may suppose that $H$ is reduced and that $\mathcal A (H)= \{u_1, \ldots, u_s\}$ is nonempty. Let $\pi \colon \mathsf Z (H) \to H$ be the factorization homomorphism. The homomorphism
\[
f \colon \mathsf Z (H) \times \mathsf Z (H) \to (\N_0^s \times \N_0^s, +),
\left(\prod_{i=1}^s u_i^{m_i}, \prod_{i=1}^s u_i^{n_i}\right) \mapsto \bigl( (m_i)_{i=1}^s, (n_i)_{i=1}^s \bigr).
\]
is an isomorphism. By Dickson's Theorem (\cite[Theorem 1.5.3]{Ge-HK06a}), every subset $A \subset \N_0^{2s}$ has only finitely many minimal points.

\smallskip
1. For every distance $d\in \Delta(H)$, we define subsets \ $R_d^+,\, R_d^- \subset \mathsf Z(H) \time \mathsf Z(H)$ as follows: let $R_d^{\pm}$ \ consist of all $(z,y)\in \mathsf Z(H) \time \mathsf Z(H)$ such that

\begin{enumerate}
\item[] $\pi(z) = \pi(y)$, \ $|z|$ and $|y|$ are adjacent lengths of $\pi(z)$, and $|y| = |z| \pm d$.
\end{enumerate}

\smallskip
\noindent
The sets $M^\pm_d = \text{\rm Min} \bigl( f (R^\pm_d)\bigr)$ of minimal points of $f (R^\pm_d)$ are finite. We set
\[
\delta^* = \max \bigl\{|z'|,\, |y'|\; \bigm| \; (z',y')\in
f^{-1}( M^+_d \cup M^-_d), \ d\in \Delta (H) \bigr\} \,,
\]
and assert that \ $\delta (z) \le \delta^*$ \ for all $z\in \mathsf Z(H)$.

Let $z\in \mathsf Z(H)$ and $k\in \N$ be such that $k$ and $|z|$ are adjacent lengths of $\pi(z)$, and let $y_0\in \mathsf Z(H)$ be any factorization with $\pi(y_0) = \pi(z)$ and $|y_0| = k$. Then we have $k = |z| \pm d$ for some $d \in \Delta (H)$ and $(z,y_0) \in R_d^\pm$. Let $(z', y') \in f^{-1}(M^\pm_d)$ be such that $f (z',y') \le f (z,y_0)$. Then $z = z' z_1$ and $y_0 = y'y_1$ for some $z_1,\, y_1 \in \mathsf Z(H)$, and we set $y = y'z_1 \in \mathsf Z(H)$. Then $\pi(y) = \pi (y') \pi (z_1) = \pi (z') \pi (z_1) = \pi (z)$,
\[
|y| = |y'| + |z_1| = |z| + |y'|-|z'| = |z| \pm d = k\,, \text{ and } \ \mathsf d(z,y) \le \max \{|z'|,|y'|\} \le \delta^*\,.
\]

\smallskip
2. Since $\mathsf c_{\adj} (H) \le \delta (H) < \infty$, it remains to show that $\mathsf c_{\eq} (H) < \infty$.
Consider the monoid
\[
S = \{(x,y)\in\mathsf Z(H)\times\mathsf Z(H)\mid |x|=|y|\ \text{and}\ \pi(xz)=\pi(yz)\text{ for some }z\in\mathsf Z(H)\}\,,
\]
and set $S^* = S \setminus \{(1,1)\}$.
First we show that $S$ is a saturated submonoid of $\mathsf Z(H)\times\mathsf Z(H)$. Clearly, $S$ is a submonoid of $\mathsf Z(H)\times\mathsf Z(H)$. Let $(x,y),(x^{\prime},y^{\prime})\in S$ and $(x^{\prime\prime},y^{\prime\prime})\in\mathsf Z(H)\times\mathsf Z(H)$ be such that $(x,y)=(x^{\prime}x^{\prime\prime},y^{\prime}y^{\prime\prime})$. There are some $z,z^{\prime}\in\mathsf Z(H)$ such that $\pi(xz)=\pi(yz)$ and $\pi(x^{\prime}z^{\prime})=\pi(y^{\prime}z^{\prime})$. Set $z^{\prime\prime}=x^{\prime}y^{\prime}z^{\prime}z$. Then $\pi(y^{\prime\prime}z^{\prime\prime})=\pi(yzx^{\prime}z^{\prime})=\pi(yz)\pi(x^{\prime}z^{\prime})=\pi(xz)\pi(y^{\prime}z^{\prime})=\pi(xzy^{\prime}z^{\prime})=\pi(x^{\prime\prime}z^{\prime\prime})$. Moreover, $|x^{\prime\prime}|=|x|-|x^{\prime}|=|y|-|y^{\prime}|=|y^{\prime\prime}|$, and thus $(x^{\prime\prime},y^{\prime\prime})\in S$.

Recall that a subset $X\subset\mathsf Z(H)$ is an $s$-ideal of $\mathsf Z(H)$ if $xy\in X$ for all $x\in X$ and $y\in\mathsf Z(H)$. Note that $\mathsf Z(H)$ and $\mathsf Z(H)\times\mathsf Z(H)$ are finitely generated reduced cancellative monoids. Consequently, $\mathsf Z(H)$ satisfies the ascending chain condition on $s$-ideals by \cite[Proposition 2.7.4]{Ge-HK06a} and $S$ is a finitely generated monoid by \cite[Proposition 2.7.5.1]{Ge-HK06a}. For $(x,y)\in S$ set
\[
{\rm A}(x,y)=\{z\in\mathsf Z(H)\mid\pi(xz)=\pi(yz)\},\ \text{and}
\]
\[
{\rm B}(x,y)=\{z\in\mathsf Z(H)\mid (x^{\prime},y^{\prime})\mid_S (x,y)\ \text{and}\ \pi(x^{\prime}z)=\pi(y^{\prime}z)\ \text{for some}\ (x^{\prime},y^{\prime})\in S^*\setminus\{(x,y)\}\}.
\]
Observe that if $(x,y),(v,w)\in S$ are such that $(v,w)\mid_S (x,y)$, then ${\rm A}(x,y)$ and ${\rm B}(x,y)$ are $s$-ideals of $\mathsf Z(H)$, ${\rm B}(v,w)\subset {\rm B}(x,y)$ and if $(1,1)\not=(v,w)\not=(x,y)$, then ${\rm A}(v,w)\subset {\rm B}(x,y)$.

\begin{enumerate}
\item[{\bf A1.}] For all $t\in\mathbb{N}$, $a\in S$ and $(b_i)_{i=1}^t\in S^t$ there is some $N\in\mathbb{N}$ such that for all $(k_i)_{i=1}^t\in\mathbb{N}_0^t$ and $j\in [1,t]$ with $k_j\geq N$ it follows that ${\rm B}(a\prod_{i=1}^t b_i^{k_i})={\rm B}(ab_j^N\prod_{i=1,i\not=j}^t b_i^{k_i})$.
\end{enumerate}

{\it Proof of} \ {\bf A1.} We proceed by induction on $t$. Let $t\in\mathbb{N}$, $a\in S$ and $(b_i)_{i=1}^t\in S^t$. Note that $({\rm B}(a\prod_{i=1}^t b_i^k))_{k\in\mathbb{N}_0}$ is an ascending chain of $s$-ideals of $\mathsf Z(H)$. Therefore, there is some $M\in\mathbb{N}$ such that for all $k\in\mathbb{N}_0$ with $k\geq M$ it follows that ${\rm B}(a\prod_{i=1}^t b_i^k)={\rm B}(a\prod_{i=1}^t b_i^M)$.

By the induction hypothesis there is some $(N_{(r,g)})_{(r,g)\in [1,t]\times [0,M-1]}\in\mathbb{N}^{[1,t]\times [0,M-1]}$ such that for all $r\in [1,t]$, $g\in [0,M-1]$, $(k_i)_{i=1}^t\in\mathbb{N}_0^t$ and $j\in [1,t]\setminus\{r\}$ such that $k_r=g$ and $k_j\geq N_{(r,g)}$ it follows that ${\rm B}(a\prod_{i=1}^t b_i^{k_i})={\rm B}(ab_j^{N_{(r,g)}}\prod_{i=1,i\not=j}^t b_i^{k_i})$. Set $N=\max\{N_{(r,g)}\mid (r,g)\in [1,t]\times [0,M-1]\}\cup\{M\}$. Let $(k_i)_{i=1}^t\in\mathbb{N}_0^t$ and $j\in [1,t]$ be such that $k_j\geq N$.

\smallskip
\noindent
CASE 1: \, $k_r \geq M$ for all $r \in [1,t]$.

Set $k=\max\{k_i\mid i\in [1,t]\}$. It follows that ${\rm B}(a\prod_{i=1}^t b_i^M)\subset {\rm B}(ab_j^N\prod_{i=1,i\not=j}^t b_i^{k_i})\subset {\rm B}(a\prod_{i=1}^t b_i^{k_i})\subset {\rm B}(a\prod_{i=1}^t b_i^k)={\rm B}(a\prod_{i=1}^t b_i^M)$, and thus ${\rm B}(a\prod_{i=1}^t b_i^{k_i})={\rm B}(ab_j^N\prod_{i=1,i\not=j}^t b_i^{k_i})$.

\smallskip
\noindent
CASE 2: \, $k_r<M$ for some $r\in [1,t]$.

Note that $k_j\geq N\geq M>k_r$, hence $j\not=r$ and $k_j\geq N\geq N_{(r,k_r)}$. We infer that ${\rm B}(a\prod_{i=1}^t b_i^{k_i})={\rm B}(ab_j^{N_{(r,k_r)}}\prod_{i=1,i\not=j}^t b_i^{k_i})$. Since ${\rm B}(ab_j^{N_{(r,k_r)}}\prod_{i=1,i\not=j}^t b_i^{k_i})\subset {\rm B}(ab_j^N\prod_{i=1,i\not=j}^t b_i^{k_i})\subset {\rm B}(a\prod_{i=1}^t b_i^{k_i})$ we have ${\rm B}(a\prod_{i=1}^t b_i^{k_i})={\rm B}(ab_j^N\prod_{i=1,i\not=j}^t b_i^{k_i})$.\qed ({\bf A1})

\smallskip
Since $S$ is finitely generated, there are some $t\in\mathbb{N}$ and $b_1, \ldots, b_t \in {S^*}$ such that $S$ is generated by $b_1, \ldots, b_t$. By ${\bf A1}$ there is some $N\in\mathbb{N}$ such that if $(k_i)_{i=1}^t\in\mathbb{N}_0^t$ and $j\in [1,t]$ are such that $k_j\geq N$, then ${\rm B}(\prod_{i=1}^t b_i^{k_i})={\rm B}(b_j^N\prod_{i=1,i\not=j}^t b_i^{k_i})$, and thus ${\rm A}(\prod_{i=1}^t b_i^{k_i})\subset {\rm B}(b_j\prod_{i=1}^t b_i^{k_i})={\rm B}(\prod_{i=1}^t b_i^{k_i})$.

This implies that $\{a\in S\mid {\rm A}(a)\not\subset {\rm B}(a)\}\subset\{\prod_{i=1}^t b_i^{k_i}\mid (k_i)_{i=1}^t\in\mathbb{N}_0^t,k_j< N$ for all $j\in [1,t]\}$, hence $\{a\in S\mid {\rm A}(a)\not\subset {\rm B}(a)\}$ is finite. Set $K=\max\{\mathsf d(x,y)\mid (x,y)\in S,{\rm A}(x,y)\not\subset {\rm B}(x,y)\}$.

\begin{enumerate}
\item[{\bf A2.}] For all $r\in\mathbb{N}_0$, $(x,y)\in S$ and $z\in {\rm A}(x,y)$ such that $|x|=r$ it follows that $xz$ and $yz$ can be concatenated by a monotone $K$-chain of factorizations of $\pi(xz)$.
\end{enumerate}

{\it Proof of} \ {\bf A2.} We proceed by induction on $r$. Let $r\in\mathbb{N}_0$, $(x,y)\in S$ and $z\in {\rm A}(x,y)$ be such that $|x|=r$.

\smallskip
\noindent
CASE 1: \, ${\rm A}(x,y)\not\subset {\rm B}(x,y)$.

We have $\mathsf d(xz,yz)=\mathsf d(x,y)\leq K$, and thus the assertion is trivially satisfied.

\smallskip
\noindent
CASE 2: \, ${\rm A}(x,y)\subset {\rm B}(x,y)$.

There are some $(x^{\prime},y^{\prime}),(x^{\prime\prime},y^{\prime\prime})\in S^*\setminus\{(x,y)\}$ such that $x=x^{\prime}x^{\prime\prime}$, $y=y^{\prime}y^{\prime\prime}$ and $\pi(x^{\prime}z)=\pi(y^{\prime}z)$. Clearly, $|x^{\prime}|<r$ and $|x^{\prime\prime}|<r$. Observe that $\pi(x^{\prime}x^{\prime\prime}z)=\pi(x^{\prime}z)\pi(x^{\prime\prime})=\pi(y^{\prime}z)\pi(x^{\prime\prime})=\pi(y^{\prime}x^{\prime\prime}z)$ and $\pi(x^{\prime\prime}y^{\prime}z)=\pi(x^{\prime\prime}x^{\prime}z)=\pi(y^{\prime\prime}y^{\prime}z)$, where the last equality holds because $z \in {\rm A} (x,y)$. We have $x^{\prime\prime}z\in {\rm A}(x^{\prime},y^{\prime})$ and $y^{\prime}z\in {\rm A}(x^{\prime\prime},y^{\prime\prime})$. Consequently, $xz=x^{\prime}x^{\prime\prime}z$ and $y^{\prime}x^{\prime\prime}z$ can be concatenated by a monotone $K$-chain of factorizations of $\pi(xz)$ by the induction hypothesis. It follows by analogy that $x^{\prime\prime}y^{\prime}z$ and $y^{\prime\prime}y^{\prime}z=yz$ can be concatenated by a monotone $K$-chain of factorizations of $\pi(x^{\prime\prime}y^{\prime}z)=\pi(xz)$. Therefore, $xz$ and $yz$ can be concatenated by a monotone $K$-chain of factorizations of $\pi(xz)$.\qed ({\bf A2})

\smallskip
It is sufficient to show that $\mathsf c_{\rm eq}(H)\leq K$. Let $a\in H$ and $x,y\in\mathsf Z(a)$ be such that $|x|=|y|$. Then $(x,y)\in S$ and $1\in {\rm A}(x,y)$. We infer by ${\bf A2}$ that $x$ and $y$ can be concatenated by a monotone $K$-chain of factorizations of $a$. Therefore, $\mathsf c_{\rm eq}(b)\leq K$ for each $b\in H$, hence $\mathsf c_{\rm eq}(H)\leq K$.
\end{proof}

\smallskip
We continue with an example (due to S. Tringali) of a finitely generated monoid whose monoid of equal-length relations is not finitely generated, although the monoid $S$, defined in the proof of Theorem \ref{3.1}.2, is finitely generated and  its equal catenary degree is finite.

\medskip
\begin{example} \label{3.2}
For an atomic monoid $H$,
\[
\sim_{H, \eq} \ = \{ (x,y) \in \mathsf Z (H) \times \mathsf Z (H) \mid \pi (x)=\pi(y) \ \text{and} \ |x|=|y| \}
\]
is the monoid of equal-length relations of $H$. If $H$ is cancellative, then $\sim_{H, \eq} \ \subset \ \sim_{H}$ is a saturated submonoid and hence a Krull monoid and it is finitely generated whenever $H_{\red}$ is finitely generated (\cite[Proposition 4.4]{Bl-Ga-Ge11a}).
We provide an example showing that $\sim_{H, \eq}$ need not be finitely generated if $H$ is not cancellative. Let
\[
\mathcal P_{{\rm fin},0}(\mathbb N_0) = \big\{ A \subset \N_0 \mid 0 \in A, \ A \subset \N_0 \ \text{is finite} \big\}
\]
be the monoid of finite subsets of $\N_0$ containing $0$, with set addition as operation. Clearly, $\{0\}$ is the zero-element of the monoid, and for every $A \subset \N_0$ and every $k \in \N$, $kA = A + \ldots + A$ means the $k$-fold sumset of $A$. Consider the submonoid $H$
generated by
\[
\{0, 1\}, \ A = \{0, 1, 3\}, \quad \text{ and} \quad B = \{0, 2, 3\} \,.
\]
 Of course, $H$ is a reduced, finitely generated, commutative semigroup with identity element, and it is unit-cancellative by \cite[Theorem 2.22(ii) and Proposition 3.3]{Fa-Tr18a}. Moreover, we have (by induction) that $\{0, 1\} + kA = \{0, 1\} + kB = [ 0, 3k+1 ]$ for all $k \in \mathbb N_0$. Since the monoid $H$ is written additively, $\mathsf Z (H) \times \mathsf Z (H)$ and $\sim_{H, \eq}$ will be written additively too.
 For every $k \in \mathbb N_0$, we define
\[
\mathfrak a_k = (\{0, 1\} + kA , \{0,1\} + kB) \in \ \sim_{H, \eq} \,,
\]
and assert that $\mathfrak a_k$ is an atom in ${\sim_{H, {\rm eq}}}$.

Assume to the contrary that there are $k \in \N$ and $\mathfrak b, \mathfrak c \in \ \sim_{H, \eq}$, distinct from the zero-element $(\{0\}, \{0\})$ of $\sim_{H, \eq}$, such that $\mathfrak a_k = \mathfrak b + \mathfrak c $. Since $\{0,1\}+ \ell A$, $\{0,1\} + \ell B$ are intervals for all $\ell \in \N$ but $\ell A$ and $\ell B$ are not intervals for any $\ell \in \N$, it follows that $(\{0,1\}, \{0,1\})$ must divide (in $\sim_{H, \eq}$) either $\mathfrak b$ or $\mathfrak c$, say it divides $\mathfrak b$. Thus we obtain that for some $q \in [0, k-1]$
\[
\mathfrak b = \big(\{0,1\}+qA, \{0,1\}+qB \big) \ \text{which implies that} \ \mathfrak c = \big( (k-q)A, (k-q)B \big) \,.
\]
However, since $k-q > 0$, we obtain that $1 \in (k-q)A$ but $1 \notin (k-q)B$, a contradiction to $\mathfrak c \in \ \sim_{H, \eq}$.
\end{example}

\medskip
The monoid of $v$-invertible $v$-ideals studied in Section \ref{5} is isomorphic to a finite direct product of finitely primary monoids and a free abelian factor. Our next goal in this section is to shift the finiteness of the monotone catenary degree of given monoids $H_1$ and $H_2$ to their direct product $H_1 \times H_2$. However, this does not hold true in general. We start with an example highlighting the problem.

\medskip
\begin{example} \label{3.3}
Let $k,d\in\N$ with $d\ge 10$, $y_{k,1},y_{k,2}\in\mathbb N_{\geq 2}$ and let
\[
L_{k,1} = y_{k,1}+\{\nu d\mid\nu\in [0,k]\}\cup\{kd+2\}\quad\text{ and}\quad L_{k,2} = y_{k,2}+\{\nu d\mid\nu\in [0,k]\}\cup\{kd+1,kd+3\}\,.
\]
We claim that both, $y_k+ kd+1$ and $y_k+kd+2$, have a unique representation in the sumset, where $y_k = y_{k,1}+y_{k,2}$. Indeed, if $y_k + kd + 1 = a_1+a_2$ with $a_1\in L_{k,1}$ and $a_2\in L_{k,2}$, then $a_1=y_{k,1}+0$ and $a_2=y_{k,2}+kd+1$.
If $y_k+ kd+2 = b_1+b_2$ with $b_1\in L_{k,1}$ and $b_2\in L_{k,2}$, then $b_1= y_{k,1}+ kd+2$ and $b_2= y_{k,2}+ 0$.
Thus $|a_1-b_1|\ge kd$ and $|a_2-b_2|\ge kd$.

Now suppose that $L_{k,1}$ and $L_{k,2}$ are realized as sets of lengths, say $L_{k,i} =\mathsf L(a_{k,i})$ with $a_{k,i}\in H_i$ for some atomic monoids $H_i$ and $i\in [1,2]$, and set $a_k = a_{k,1}a_{k,2}$ (there is a variety of realization results for sets of lengths guaranteeing the existence of $H_1$ and $H_2$; see, for example, \cite[Proposition 4.8.3]{Ge-HK06a}, \cite[Theorem 4.2]{Ge-Ha-Le07}, \cite{Sc09a}). Then
\[
\mathsf d\big(\mathsf Z_{y_k+kd+1}(a_k),\mathsf Z_{y_k+kd+2} (a_k)\big)\ge 2kd\,.
\]
\end{example}

\smallskip
For a further analysis of the problem, consider an atomic monoid $H$. If $a \in H$ and $k, \ell \in \mathsf L (a)$ are adjacent lengths with $k < \ell$, then by definition we have $\mathsf d \big( \mathsf Z_k (a), \mathsf Z_{\ell} (a) \big) \le \mathsf c_{\adj}(a) \le \mathsf c_{\adj} (H)$. However, if $\ell, m \in \mathsf L (a)$ are adjacent with $\ell < m$, then we cannot conclude that
\[
\mathsf d \big( \mathsf Z_k (a), \mathsf Z_{m} (a) \big) \le \mathsf d \big( \mathsf Z_k (a), \mathsf Z_{\ell} (a) \big) + \mathsf d \big( \mathsf Z_{\ell} (a), \mathsf Z_{m} (a) \big) \,.
\]
In order to be able to obtain an inductive upper bound for $\mathsf d \big( \mathsf Z_{k_0} (a), \mathsf Z_{k_r} (a) \big)$ for pairwise adjacent lengths $k_0 < \ldots < k_r$ in $\mathsf L (a)$, we introduce a new invariant.

\smallskip
\begin{definition} \label{3.4}
Let $H$ be an atomic monoid.
For an element $a \in H$, the {\it $($weak$)$ successive distance} $\delta_w (a)$ is the smallest $N \in \N_0 \cup \{\infty\}$ such that for all $k, \ell \in \mathsf L (a)$ we have $\mathsf d \big( \mathsf Z_k (a), \mathsf Z_{\ell} (a) \big) \le N | \ell - k|$. Then
\[
\delta_w (H) = \sup \{ \delta_w (a) \mid a \in H \} \in \N_0 \cup \{\infty\}
\]
is called the {\it $($weak$)$ successive distance} of $H$.
\end{definition}

\smallskip
Note that $\delta_w (H) = 0$ if and only if $H$ is half-factorial.
The invariant $\delta_w (H)$ is bounded above by $\delta (H)$ (as we outline in the next lemma). However, in Section \ref{5} we will meet monoids having finite weak successive distance but infinite strong successive distance (see Corollary \ref{5.7}, Proposition \ref{5.12}, and Remark \ref{5.3}.4).

\smallskip
\begin{lemma} \label{3.5}
Let $H$ be an atomic monoid.
\begin{enumerate}
\item $\delta_w (H) \le \delta (H)$.

\smallskip
\item Suppose that $\Delta(H)$ is finite. Then $\mathsf c_{\adj}(H)\le\delta_w (H)\max\Delta(H)$. In particular, if $\delta_w(H)<\infty$ and $\mathsf c_{\eq}(H) < \infty$, then $\mathsf c_{\mon}(H)<\infty$.
\end{enumerate}
\end{lemma}

\begin{proof}
If $H$ is half-factorial, then $\Delta(H)=\emptyset$ and $\mathsf c_{\adj}(H)=\max\Delta(H)=\delta_w(H)=0$ whence all assertions hold. Now we suppose that $H$ is not half-factorial.

\smallskip
1. Clearly, it is sufficient to show that $\delta_w(a)\le\delta(H)$ for every $a \in H$. Let $a\in H$ and $k, \ell \in \mathsf L (a)$ be given, say $k < \ell$. Then there are pairwise adjacent lengths $k=k_0 < k_1 < \ldots < k_r = \ell$ of $\mathsf L (a)$. By definition of $\delta (H)$, there are factorizations $z_0, \ldots, z_r$ such that $|z_i|=k_i$ for $i \in [0,r]$ and $\mathsf d (z_{i-1}, z_i) \le \delta (H)$ for all $i \in [1,r]$. Therefore, it follows that
\[
\mathsf d \big( \mathsf Z_{k_0} (a), \mathsf Z_{k_r} (a) \big) \le \mathsf d (z_0, z_{r}) \le \sum_{i=1}^r \mathsf d (z_{i-1},z_i) \le r \delta (H) \le |\ell - k| \delta (H) \,.
\]

\smallskip
2. If $a \in H$ and $k, \ell \in \mathsf L (a)$ are adjacent, then $\mathsf d \big( \mathsf Z_k (a), \mathsf Z_{\ell} (a) \big) \le \delta_w (a) \max \Delta \big( \mathsf L (a) \big)$ whence $\mathsf c_{\adj} (a) \le \delta_w (a) \max \Delta \big( \mathsf L (a) \big)$. Thus it follows that $\mathsf c_{\adj} (H) \le \delta_w (H) \max \Delta (H)$, and the in particular statement follows immediately.
\end{proof}

\medskip
\begin{lemma}\label{3.6}
Let $H$ be an atomic monoid satisfying the Structure Theorem for Sets of Lengths and suppose that the following conditions hold{\rm \,:}
\begin{itemize}
\item[(a)] There are some $M_1,C_1\in\mathbb{N}$ such that for each $a\in H$ and all adjacent $k, \ell\in\mathsf L(a)$ such that $\max\{k,\ell\}+M_1\leq\max\mathsf L(a)$ it follows that ${\rm Dist}(\mathsf Z_k(a),\mathsf Z_{\ell} (a))\leq C_1$.

\smallskip
\item[(b)] For every $N\in\mathbb{N}$ there is some $C_2\in\mathbb{N}$ such that for all $a\in H$ and $k, \ell\in\mathsf L(a)$ for which $\min\{k, \ell\}+N\geq\max\mathsf L(a)$ it follows that $\mathsf d(\mathsf Z_k(a),\mathsf Z_{\ell} (a))\leq C_2|\ell - k|$.
\end{itemize}
Then $\delta_w (H) < \infty$.
\end{lemma}

\begin{proof}
Without restriction we may suppose that $H$ be reduced. We start with the following assertion.

\begin{enumerate}
\item[{\bf A.}\,] There is some $M_0\in\mathbb{N}$ such that for all $a\in H$ and $k\in\mathsf L(a)$ with $k+2M_0\leq\max\mathsf L(a)$ we have $k+M_0\in\mathsf L(a)$.
\end{enumerate}

{\it Proof of} \ {\bf A.} Since $H$ satisfies the Structure Theorem for Sets of Lengths there is some $N\in\mathbb{N}$ and some finite nonempty $\Delta\subset\mathbb{N}$ such that for all $a\in H$, $\mathsf L(a)$ is an AAMP with bound $N$ and difference in $\Delta$. Set $M_0=N\prod_{d\in\Delta} d$. Let $a\in H$ and $k\in\mathsf L(a)$ be such that $k+2M_0\leq\max\mathsf L(a)$. There is some $d\in\Delta$ such that $\mathsf L(a)$ is an AAMP with bound $N$ and difference $d$. Note that $\min\mathsf L(a)+M_0\leq k+M_0\leq\max\mathsf L(a)-M_0$. Since $k\in\mathsf L(a)$ and $M_0\in d\mathbb{Z}$ we infer that $k+M_0\in\mathsf L(a)$. \qed ({\bf A})

\smallskip
Set $M=M_0M_1$. It follows easily by induction that for all $a\in H$ and $k\in\mathsf L(a)$ such that $k+2M\leq\max\mathsf L(a)$ it follows that $k+M\in\mathsf L(a)$. There is some $C_2\in\mathbb{N}$ such that for all $k, \ell\in\mathsf L(a)$ for which $\min\{k, \ell\}+2M\geq\max\mathsf L(a)$ it follows that $\mathsf d(\mathsf Z_k(a),\mathsf Z_{\ell}(a))\leq C_2|\ell - k|$.

We assert that $\delta_w (H) \le \max\{C_1,C_2\}$.
Let $a\in H$ and $k,\ell\in\mathsf L(a)$, say $k<\ell$. Let $\mathsf L (a) \cap [k, \ell] = \{k_0, \ldots, k_r\}$ with $k=k_0 < \ldots < k_r = \ell$.

\smallskip
\noindent
CASE 1: \, $\ell+M\leq\max\mathsf L(a)$.

Note that ${\rm Dist}(\mathsf Z_{k_i}(a),\mathsf Z_{k_{i+1}}(a))\leq C_1$ for all $i\in [0,r-1]$. By definition, there are factorizations $z_i\in\mathsf Z_{k_i}(a)$ for every $i\in [0,r]$ such that $\mathsf d(z_j,z_{j+1})=\mathsf d(\{z_j\},\mathsf Z_{k_{j+1}}(a))$ for all $j\in [0,r-1]$ whence we infer that
\[
\mathsf d(\mathsf Z_k(a),\mathsf Z_\ell(a))\leq\mathsf d(z_0,z_r)\leq\sum_{j=0}^{r-1}\mathsf d(z_j,z_{j+1})\leq C_1r\leq \max\{C_1,C_2\}|\ell - k| \,.
\]

\smallskip
\noindent
CASE 2: \, $k+2M\geq\max\mathsf L(a)$.

It is clear that $\mathsf d(\mathsf Z_k(a),\mathsf Z_\ell(a))\leq C_2|\ell - k|\leq \max\{C_1,C_2\}|\ell - k|$.

\smallskip
\noindent
CASE 3: \, $\ell+M>\max\mathsf L(a)$ and $k+2M<\max\mathsf L(a)$.

There is some maximal $n\in [0,r]$ such that $k_n+2M\leq\max\mathsf L(a)$. Observe that $k_n+M\in\mathsf L(a)$. Therefore, there is some $m\in [1,r-1]$ such that $k_m=k_n+M$. There are factorizations $z_0, \ldots, z_{m+1}$ such that $z_0\in\mathsf Z_\ell(a)$, $z_i\in\mathsf Z_{k_{m+1-i}}(a)$ for every $i\in [1,m+1]$, $\mathsf d(z_0,z_1)=\mathsf d(\mathsf Z_{k_m}(a),\mathsf Z_\ell(a))$ and $\mathsf d(z_j,z_{j+1})=\mathsf d(\{z_j\},\mathsf Z_{k_{m-j}}(a))$ for every $j\in [1,m]$. If $j\in [1,m]$, then $k_{m-j+1}+M\leq\max\mathsf L(a)$, and thus $\mathsf d(z_j,z_{j+1})\leq C_1$. Moreover, since $k_m+2M\geq\max\mathsf L(a)$ we have $\mathsf d(z_0,z_1)\leq C_2(\ell-k_m)$. This implies that
\[
\mathsf d(\mathsf Z_k(a),\mathsf Z_\ell(a))\leq\mathsf d(z_0,z_{m+1})\leq\sum_{i=0}^m\mathsf d(z_i,z_{i+1})\leq C_2(\ell-k_m)+C_1m\leq \max\{C_1,C_2\}|\ell - k| \,.
\]
\end{proof}

\medskip
\begin{lemma} \label{3.7}
Let $H = H_1 \times \ldots \times H_n$ where $n \in \N$ and $H_1, \ldots, H_n$ are atomic monoids.
\begin{enumerate}
\item $\Delta (H)$ is finite if and only if $\Delta (H_i)$ is finite for every $i \in [1,n]$.

\smallskip
\item The Structure Theorem for Sets of Lengths holds for $H$ if and only if it holds for $H_1, \ldots, H_n$.
\end{enumerate}
\end{lemma}

\begin{proof}
Without restriction we may suppose that $H$ is reduced. Then $H_1, \ldots, H_n$ are reduced and they are divisor-closed submonoids of $H$. Thus, if $\Delta (H)$ is finite (or if the Structure Theorem for Sets of Lengths holds for $H$, then the same is true for each divisor-closed submonoid.

If $\Delta (H_1), \ldots, \Delta (H_n)$ are finite, then $\Delta (H)$ is finite (a proof in the cancellative setting can be found in \cite[Proposition 1.4.5]{Ge-HK06a}, and the proof in the general setting runs along the same lines). Clearly, we have
\[
\mathcal L (H) = \{L_1+ \ldots + L_n \mid L_i \in \mathcal L (H_i) \ \text{for all} \ i \in [1,n] \}
\]
whence sets of lengths in $H$ are sumsets of sets of lengths in $H_1, \ldots, H_n$. Thus if the Structure Theorem for Sets of Lengths holds for $H_1, \ldots, H_n$, then it holds for $H$ by \cite[Theorem 4.2.16]{Ge-HK06a}.
\end{proof}

\medskip
\begin{theorem} \label{3.8}
Let $H = H_1 \times \ldots \times H_n$ where $n \in \N$ and $H_1, \ldots, H_n$ are atomic monoids. Then the following statements are equivalent{\rm \,:}
\begin{enumerate}
\item[(a)] $H$ satisfies the Structure Theorem for Sets of Lengths, $\mathsf c_{\eq} (H) < \infty$, and $\delta_w (H) < \infty$.

\smallskip
\item[(b)] For every $i \in [1,n]$, $H_i$ satisfies the Structure Theorem for Sets of Lengths, $\mathsf c_{\eq} (H_i) < \infty$, and $\delta_w (H_i) < \infty$.
\end{enumerate}
\end{theorem}

\begin{proof}
We may suppose that $H$ is reduced. Then $H_1, \ldots, H_n$ are reduced divisor-closed submonoids of $H$ and hence (a) implies (b). In order to show that (b) implies (a) it suffices to handle the case $n=2$.

Therefore, we suppose that $H_1$ and $H_2$ satisfy the Structure Theorem for Sets of Lengths and that $\mathsf c_{\eq} (H_1)$, $\mathsf c_{\eq} (H_2)$, $\delta_w (H_1)$, and $\delta_w (H_2)$ are all finite. Lemma \ref{3.7} implies that the Structure Theorem for Sets of Lengths holds for $H$. Now we proceed in two steps.

\smallskip
1. We show that $\delta_w (H_1 \times H_2) < \infty$. For every $i \in [1,2]$,
there is a finite nonempty set $\Delta_i$ and a bound $M_i \in\mathbb{N}$ such that for every $a_i \in H_i$, $\mathsf L(a_i)$ is an AAMP with bound $M_i$ and difference in $\Delta_i$ and for all $k, \ell\in\mathsf L(a_i)$, we have $\mathsf d(\mathsf Z_k(a_i),\mathsf Z_{\ell}(a_i))\leq {\delta_w (H_i)}|\ell - k|$.

We set $e=\max\{M_1,M_2\}\prod_{d\in\Delta_1\cup\Delta_2} d$, and assert that
\[
\delta_w (H_1 \times H_2) \le \max\{{\delta_w (H_1)},{\delta_w (H_2)}\}(4e+1) \,.
\]
Let $a\in H$ and $k, \ell \in\mathsf L(a)$ be distinct. Then there are some $a_1 \in H_1$ and $a_2 \in H_2$ such that $a= a_1a_2$. Obviously, there are some $k_1,\ell_1\in\mathsf L (a_1)$ and $k_2,\ell_2\in\mathsf L (a_2)$ such that $k=k_1+k_2$, $\ell =\ell_1+\ell_2$ and for all $v_1,w_1 \in\mathsf L (a_1)$ and $v_2,w_2\in\mathsf L (a_2)$ with $k=v_1+v_2$ and $\ell =w_1+w_2$ it follows that $|\ell_1 - k_1|+|\ell_2-k_2|\leq |v_1-w_1|+|v_2-w_2|$.
We proceed by proving the following assertion.

\begin{enumerate}
\item[{\bf A.}\,] $|\ell_1 - k_1|+|\ell_2 - k_2|\leq (4e+1)|\ell - k|$.
\end{enumerate}

{\it Proof of} \ {\bf A.} We distinguish two cases.

\smallskip
\noindent
CASE 1: \, $|\ell_1 - k_1|\leq 2e$ or $|\ell_2 - k_2|\leq 2e$.

Without restriction we may suppose that $|\ell_1 - k_1|\leq 2e$.
Since $k_2-\ell_2=k- \ell-(k_1-\ell_1)$ we infer that $|\ell_2 - k_2|\leq |\ell - k|+|\ell_1 - k_1|\leq |\ell - k|+2e\leq (2e+1)|\ell - k|$. It follows that $|\ell_1 - k_1|+|\ell_2 - k_2|\leq 2e+(2e+1)|\ell - k|\leq(4e+1)|\ell - k|$.

\smallskip
\noindent
CASE 2: \, $|\ell_1 - k_1|>2e$ and $|\ell_2 - k_2|>2e$.

First, suppose that $k_1\leq \ell_1$ and $k_2\geq \ell_2$. We set
\[
v_1=k_1+e, \ w_1=\ell_1-e, \ v_2=k_2-e, \quad \text{ and} \quad w_2=\ell_2+e \,,
\]
and observe that
\[
\min\mathsf L (a_1)+M_1\leq v_1\leq w_1\leq\max\mathsf L (a_1)-M_1 \quad \text{and} \quad \min\mathsf L (a_2)+M_2\leq w_2\leq v_2\leq\max\mathsf L (a_2)-M_2 \,.
\]
Since $e$ is a multiple of the differences of the AAMPs $\mathsf L (a_1)$ and $\mathsf L (a_2)$, it follows that $v_1,w_1\in\mathsf L (a_1)$ and $v_2,w_2\in\mathsf L (a_2)$. Therefore, since $k=v_1+v_2$ and $\ell =w_1+w_2$, our minimality choice of $k_1,k_2,\ell_1$, and $\ell_2$ implies that
\[
\begin{aligned}
|\ell_1 - k_1|+|\ell_2 - k_2| & \leq |v_1-w_1|+|v_2-w_2|=|k_1-\ell_1+2e|+|k_2-\ell_2-2e| \\
 & =\ell_1-k_1-2e+k_2-\ell_2-2e=|\ell_1 - k_1|+|\ell_2 - k_2|-4e \,,
\end{aligned}
\]
 a contradiction.

Second, if $k_1\geq \ell_1$ and $k_2\leq \ell_2$, then we again obtain a contradiction. Thus we infer that ($k_1\leq \ell_1$ and $k_2\leq \ell_2$) or ($k_1\geq \ell_1$ and $k_2\geq \ell_2$). In both cases it is obvious that
\[
|\ell_1 - k_1|+|\ell_2 - k_2|=|\ell - k|\leq (4e+1)|\ell - k| \,. \qed ({\bf A})
\]

\smallskip
Clearly, there are some $x_1\in\mathsf Z_{k_1}(a_1)$, $y_1\in\mathsf Z_{\ell_1}(a_1)$, $x_2\in\mathsf Z_{k_2}(a_2)$, and $y_2\in\mathsf Z_{\ell_2}(a_2)$ such that $\mathsf d(x_1,y_1)=\mathsf d(\mathsf Z_{k_1}(a_1),\mathsf Z_{\ell_1}(a_1))$ and $\mathsf d(x_2,y_2)=\mathsf d(\mathsf Z_{k_2}(a_2),\mathsf Z_{\ell_2}(a_2))$. Therefore, {\bf A} implies that
\[
\begin{aligned}
\mathsf d(\mathsf Z_k(a),\mathsf Z_{\ell}(a)) & \leq \mathsf d(x_1x_2,y_1y_2)\leq\mathsf d(x_1,y_1)+\mathsf d(x_2,y_2) \\
 & \leq {\delta_w (H_1)}|\ell_1 - k_1|+{\delta_w (H_2)}|\ell_2 - k_2| \\
 & \leq\max\{{\delta_w (H_1)},{\delta_w (H_2)}\}(|\ell_1 - k_1|+|\ell_2 - k_2|) \\ & \leq \max\{{\delta_w (H_1)},{\delta_w (H_2)}\}(4e+1) |\ell - k| \,.
\end{aligned}
\]

\bigskip
2. We assert that $\mathsf c_{\eq} (H_1 \times H_2) \le N$, where
\[
N=\max\{2e({\delta_w (H_1)}+{\delta_w (H_2)}),\mathsf c_{\eq}(H_1),\mathsf c_{\eq}(H_2)\} \,.
\]
To prove that $\mathsf c_{\eq}(H_1\times H_2)\leq N$ it is sufficient to show that for all $a_1 \in H_1$, $a_2 \in H_2$, $r\in\mathbb{N}_0$, $z_1 ,v_1 \in\mathsf Z(a_1)$, $z_2,v_2\in\mathsf Z(a_2)$ such that $|v_1|-|z_1|=r$ and $|z_1|+|z_2|=|v_1|+|v_2|$ it follows that $z_1z_2$ and $v_1v_2$ can be concatenated by a monotone $N$-chain. Let $a_1 \in H_1$ and $a_2 \in H_2$. We prove the assertion by induction on $r$. Let $r\in\mathbb{N}_0$ be such that for every $s\in\mathbb{N}_0$ with $s<r$ and all $z_1,v_1\in\mathsf Z(a_1)$, $z_2,v_2\in\mathsf Z(a_2)$ such that $|v_1|-|z_1|=s$ and $|z_1|+|z_2|=|v_1|+|v_2|$ it follows that $z_1z_2$ and $v_1v_2$ can be concatenated by a monotone $N$-chain. Let $z_1,v_1 \in\mathsf Z(a_1)$, $z_2,v_2\in\mathsf Z(a_2)$ be such that $|v_1|-|z_1|=r$ and $|z_1|+|z_2|=|v_1|+|v_2|$. Obviously, $|z_2|-|v_2|=|v_1|-|z_1|$. We distinguish two cases.

\smallskip
\noindent
CASE 1: \, $|z_2|-|v_2|\leq 2e$.

Observe that $\mathsf d(\mathsf Z_{|z_1|}(a_1),\mathsf Z_{|v_1|}(a_1))\leq 2e{\delta_w (H_1)}$ and $\mathsf d(\mathsf Z_{|z_2|}(a_2),\mathsf Z_{|v_2|}(a_2))\leq 2e{\delta_w (H_2)}$. Note that there are some $x_1\in\mathsf Z_{|z_1|}(a_1)$, $y_1\in\mathsf Z_{|v_1|}(a_1)$, $x_2\in\mathsf Z_{|z_2|}(a_2)$ and $y_2\in\mathsf Z_{|v_2|}(a_2)$ such that $\mathsf d(x_1,y_1)=\mathsf d(\mathsf Z_{|z_1|}(a_1),\mathsf Z_{|v_1|}(a_1))$ and $\mathsf d(x_2,y_2)=\mathsf d(\mathsf Z_{|z_2|}(a_2),\mathsf Z_{|v_2|}(a_2))$. There is some monotone $\mathsf c_{\eq}(H_1)$-chain which concatenates $z_1$ and $x_1$, and thus there is some monotone $\mathsf c_{\eq}(H_1)$-chain which concatenates $z_1z_2$ and $x_1z_2$. Moreover, there is some monotone $\mathsf c_{\eq}(H_2)$-chain which concatenates $z_2$ and $x_2$, and thus there is some monotone $\mathsf c_{\eq}(H_2)$-chain which concatenates $x_1z_2$ and $x_1x_2$. Therefore, there is some monotone $N$-chain which concatenates $z_1z_2$ and $x_1x_2$. Along the same lines one can show that there is some monotone $N$-chain which concatenates $y_1y_2$ and $v_1v_2$. Observe that $\mathsf d(x_1x_2,y_1y_2)\leq\mathsf d(x_1,y_1)+\mathsf d(x_2,y_2)\leq 2e({\delta_w (H_1)}+{\delta_w (H_2)})\leq N$. Therefore, there is some monotone $N$-chain which concatenates $z_1z_2$ and $v_1v_2$.

\smallskip
\noindent
CASE 2: \, $|z_2|-|v_2|> 2e$.

Note that $|v_1|-|z_1|>2e$, $\min\mathsf L (a_1)+M_1\leq |z_1|+e\leq\max\mathsf L (a_1)-M_1$ and $\min\mathsf L (a_2)+M_2\leq |z_2|-e\leq\max\mathsf L (a_2)-M_2$. Since $e$ is a multiple of the differences of the AAMPs $\mathsf L (a_1)$ and $\mathsf L (a_2)$ we have $|z_1|+e\in\mathsf L (a_1)$ and $|z_2|-e\in\mathsf L (a_2)$. There are some $w_1\in\mathsf Z(a_1)$ and $w_2\in\mathsf Z(a_2)$ such that $|w_1|=|z_1|+e$ and $|w_2|=|z_2|-e$. It follows that $0\leq |v_1|-|w_1|=r-e<r$ and $|w_1|+|w_2|=|z_1|+|z_2|=|v_1|+|v_2|$. By the induction hypothesis, $w_1w_2$ and $v_1v_2$ can be concatenated by a monotone $N$-chain. Note that there are some $x_1\in\mathsf Z_{|z_1|}(a_1)$, $y_1\in\mathsf Z_{|w_1|}(a_1)$, $x_2\in\mathsf Z_{|z_2|}(a_2)$, $y_2\in\mathsf Z_{|w_2|}(a_2)$ such that $\mathsf d(x_1,y_1)=\mathsf d(\mathsf Z_{|z_1|}(a_1),\mathsf Z_{|w_1|}(a_1))$ and $\mathsf d(x_2,y_2)=\mathsf d(\mathsf Z_{|z_2|}(a_2),\mathsf Z_{|w_2|}(a_2))$. There is some monotone $\mathsf c_{\eq}(H_1)$-chain which concatenates $z_1$ and $x_1$, and thus there is some monotone $\mathsf c_{\eq}(H_1)$-chain which concatenates $z_1z_2$ and $x_1z_2$. Moreover, there is some monotone $\mathsf c_{\eq}(H_2)$-chain which concatenates $z_2$ and $x_2$, and thus there is some monotone $\mathsf c_{\eq}(H_2)$-chain which concatenates $x_1z_2$ and $x_1x_2$. Therefore, there is some monotone $N$-chain which concatenates $z_1z_2$ and $x_1x_2$. Along the same lines one can show that there is some monotone $N$-chain which concatenates $y_1y_2$ and $w_1w_2$. Observe that $\mathsf d(x_1x_2,y_1y_2)\leq\mathsf d(x_1,y_1)+\mathsf d(x_2,y_2)\leq e({\delta_w (H_1)}+{\delta_w (H_2)})\leq N$. Consequently, there is some monotone $N$-chain which concatenates $z_1z_2$ and $w_1w_2$, and thus there is some monotone $N$-chain which concatenates $z_1z_2$ and $v_1v_2$.
\end{proof}

\medskip
\section{Finitely generated monoids of $r$-ideals} \label{4}
\medskip

Our notation of ideal theory follows \cite{HK98} with the modifications stemming from the fact that the monoids in this paper do not contain a zero-element. We recall some basics in the setting of commutative monoids and will use all notations also for integral domains.

Let $H$ be a cancellative monoid. An {\it ideal system} on $H$ is a map $r\colon\mathcal P(H)\to\mathcal P(H)$ such that the following conditions are satisfied for all subsets $X,Y\subset H$ and all $c\in H$:
\begin{itemize}
\item $X\subset X_r$.

\item $X\subset Y_r$ implies $X_r\subset Y_r$.

\item $cH\subset\{c\}_r$.

\item $cX_r=(cX)_r$.
\end{itemize}
Let $r$ be an ideal system on $H$. A subset $I \subset H$ is called an $r$-ideal if $I_r=I$. We denote by $\mathcal I_r (H)$ the set of all nonempty $r$-ideals, and we define $r$-multiplication by setting $I \cdot_r J = (IJ)_r$ for all $I,J \in \mathcal I_r (H)$. Then $\mathcal I_r (H)$ together with $r$-multiplication is a reduced semigroup with identity element $H$. Let $\mathcal F_r (H)$ denote the semigroup of fractional $r$-ideals, $\mathcal F_r (H)^{\times}$ the group of $r$-invertible fractional $r$-ideals, and $\mathcal I_r^*(H) = \mathcal F_r^{\times} (H) \cap \mathcal I_r (H)$ the cancellative monoid of $r$-invertible $r$-ideals of $H$ with $r$-multiplication. We say that $H$ satisfies the $r$-{\it Krull Intersection Theorem} if
\[
\bigcap_{n\in\mathbb{N}_0} (I^n)_{r} = \emptyset \quad \text{ for all} \quad I\in\mathcal{I}_{r}(H) \setminus \{H\} \,.
\]
For subsets $A, B \subset \mathsf q(H)$, we denote by $(A\DP B)=\{x\in\mathsf q(H)\mid xB\subset A\}$, by $A^{-1}=(H \DP A)$, by $A_v=(A^{-1})^{-1}$ and by $A_t=\bigcup_{E\subset A,|E|<\infty} E_v$.
We will use the $s$-system, the $v$-system, and the $t$-system. By $\mathfrak X (H)$, we denote the set of minimal nonempty prime $s$-ideals of $H$.
If $r$ and $q$ are ideal systems on $H$, then we say that $r$ is finer than $q$ (resp. $q$ is coarser than $r$) if $\mathcal I_q(H)\subset\mathcal I_r(H)$ (equivalently, $X_r\subset X_q$ for all $X\subset H$). Recall that $X_r\subset X_v$, and if $r$ is finitary, then $X_r\subset X_t$ for all $X \subset H$.

We need the following classes of monoids (see \cite[Section 2.10]{Ge-HK06a}, also for the correspondence to ring theory). Let $H$ be a monoid and $\mathfrak m=H\setminus H^{\times}$. Then $H$ is said to be
\begin{itemize}
\item {\it archimedean} if $\cap_{n\ge 0} a^nH = \emptyset$ for every $a\in\mathfrak m$,

\item {\it primary} if $H\ne H^{\times}$ and $s$-$\spec(H)=\{\emptyset,\mathfrak m\}$,

\item {\it strongly primary} if for every $a\in\mathfrak m$ there is an $n\in\N$ such that $\mathfrak m^n\subset aH$,

\item a {\it $\G$-monoid} if the intersection of all nonempty prime $s$-ideals is nonempty.
\end{itemize}
To recall some well-known properties, note that every monoid satisfying the ACCP is archimedean, that every $v$-Noetherian primary monoid is strongly primary, that every primary monoid is a G-monoid, and that overmonoids of G-monoids are G-monoids.

In order to give the definition of C-monoids we need to recall the concept of class semigroups which are a refinement of ordinary class groups in commutative algebra (a detailed presentation can be found in \cite[Chapter 2]{Ge-HK06a}).
Let $F$ be a cancellative monoid and $H \subset F$ a submonoid. Two elements $y, y' \in F$
are called $H$-equivalent, if $y^{-1}H \cap F = {y'}^{-1} H \cap
F$. $H$-equivalence is a congruence relation on $F$. For $y \in
F$, let $[y]_H^F$ denote the congruence class of $y$, and let
\[
\mathcal C(H,F)=\{[y]_H^F\mid y\in F\}\quad\text{and}
\quad\mathcal C^*(H,F)=\{ [y]_H^F \mid y \in (F\setminus
F^{\times})\cup\{1\}\}.
\]
Then $\mathcal C (H,F)$ is a semigroup with unit element $[1]_H^F$
(called the {\it class semigroup} of $H$ in $F$) and $\mathcal C^*(H,F)
\subset \mathcal C (H,F)$ is a subsemigroup (called the reduced
class semigroup of $H$ in $F$). If $H$ is a submonoid of a factorial monoid $F$
      such that $H \cap F^\times = H^\times$ and $\mathcal{C}^*(H,F)$ is finite,
      then $H$ is called a {\it {\rm C}-monoid}.
      A C-monoid $H$ is $v$-Noetherian with $(H \colon \widehat H) \ne \emptyset$ and its complete integral closure $\widehat H$ is Krull with finite class group (\cite[Theorems 2.9.11 and 2.9.13]{Ge-HK06a}).
      If $H$ is a Krull monoid with finite class group, then $H$ is a C-monoid and the class semigroup coincides with the usual class group of a Krull monoid.

Let $R$ be a domain with quotient field $K$ and $r$ an ideal system on $R$ (clearly, $R^{\bullet}$ is a monoid and $r$ restricts to an ideal system $r'$ on $R^{\bullet}$ whence for every subset $I \subset R$ we have $I_r = (I^{\bullet})_{r'} \cup \{0\}$). We denote by $\mathcal I_r(R)$ the semigroup of nonzero $r$-ideals of $R$ and $\mathcal I_r^*(R) \subset \mathcal I_r(R)$ is the subsemigroup of $r$-invertible $r$-ideals of $R$. The usual ring ideals form an ideal system, called the $d$-system, and for these ideals we omit all suffices (i.e., $\mathcal I(R) = \mathcal I_d(R)$, and so on).

\smallskip
\begin{center} {\it Throughout the rest of the paper, every ideal system $r$ on a domain $R$ \\ has the property that $\mathcal I_r(R) \subset \mathcal I(R)$.}
\end{center}
\smallskip

A domain $R$ is said to be archimedean (primary, strongly primary, a G-domain, a C-domain) if its multiplicative monoid $R^{\bullet}$ has the respective property. Recall that a domain is primary if and only if it is one-dimensional local. To give an example for a C-domain, let $R$ be a Mori domain with nonzero conductor $\mathfrak f =(R\colon\widehat R)$. If the class group $\mathcal C(\widehat R)$ and the factor ring $\widehat R/\mathfrak f$ are both finite, then $R$ is a C-domain by \cite[Theorem 2.11.9]{Ge-HK06a} (for more on C-domains see \cite{Ge-Ra-Re15c,Re13a}).

\medskip
In the first part of this section we establish sufficient criteria for the $r$-ideal semigroup $\mathcal{I}_{r}(R)$ to be unit-cancellative.

\begin{lemma}\label{4.1}
Let $R$ be a domain and $r$ an ideal system on $R$.
\begin{enumerate}
\item If $R$ satisfies the $r$-Krull Intersection Theorem, then $\mathcal{I}_r(R)$ is unit-cancellative.

\smallskip
\item $\mathcal{I}_r(R)$ is unit-cancellative if and only if $\mathcal{I}_r(R)$ has no nontrivial idempotents and for every $I\in\mathcal{I}_r(R)$, $\{J\in\mathcal{I}_r(R)\mid (IJ)_r=I\}$ has a minimal element with respect to inclusion.
\end{enumerate}
\end{lemma}

\begin{proof}
1. Let $I,J\in\mathcal{I}_r(R)$ be such that $(IJ)_r=I$. We infer that $(IJ^n)_r=I$ for all $n\in\mathbb{N}_0$, hence $\{0\}\not=I\subset\bigcap_{n\in\mathbb{N}_0} (J^n)_r$. Since $R$ satisfies the $r$-Krull Intersection Theorem we have $J=R$.

2. First let $\mathcal{I}_r(R)$ be unit-cancellative. It is clear that $R$ is the only idempotent element of $\mathcal{I}_r(R)$ and for every $I\in\mathcal{I}_r(R)$, $\{J\in\mathcal{I}_r(R)\mid (IJ)_r=I\}=\{R\}$.

To prove the converse, suppose that $\mathcal{I}_r(R)$ has no nontrivial idempotents and let for every $I\in\mathcal{I}_r(R)$,\linebreak $\{B\in\mathcal{I}_r(R)\mid (IB)_r=I\}$ have a minimal element with respect to inclusion. Let $I,J\in\mathcal{I}_r(R)$ be such that $(IJ)_r=I$. We have to show that $J=R$. Let $A\in\mathcal{I}_r(R)$ be minimal such that $(IA)_r=I$. Since $(IA^2)_r=I$ and $(A^2)_r\subset A$, we infer that $A$ is an idempotent of $\mathcal{I}_r(R)$, and thus $A=R$. Note that $(IJA)_r=I$ and $(JA)_r\subset A$. Therefore, $J=(JA)_r=A=R$.
\end{proof}

\medskip
\begin{proposition}\label{4.2}
Let $R$ be a domain and $r$ an ideal system on $R$.
\begin{enumerate}
\item If $R$ is strongly primary, then $R$ satisfies the $r$-Krull Intersection Theorem.

\smallskip
\item Let $R$ be archimedean such that $\widehat{R}$ is a semilocal principal ideal domain. If $(R:\widehat{R})\not=\{0\}$ or $r$ is finitary, then $\mathcal{I}_r(R)$ is unit-cancellative.

\smallskip
\item If $R$ is an $r$-Noetherian $\G$-domain, then $\mathcal{I}_r(R)$ is unit-cancellative.
\end{enumerate}
\end{proposition}

\begin{proof}
1. Set $\mathfrak m=R\setminus R^{\times}$. Let $I\in\mathcal{I}_r(R)$ be such that $\bigcap_{n\in\mathbb{N}_0} (I^n)_r\not=\{0\}$. There is some nonzero $x\in\bigcap_{n\in\mathbb{N}_0} (I^n)_r$. We infer that $\mathfrak m^k\subset xR$ for some $k\in\mathbb{N}_0$, and thus $(\mathfrak m^k)_r\subset xR$. Assume that $I\not=R$. Then $I\subset\mathfrak m$, hence $(\mathfrak m^k)_r\subset xR\subset\bigcap_{n\in\mathbb{N}_0} (I^n)_r\subset\bigcap_{n\in\mathbb{N}_0} (\mathfrak m^n)_r\subset (\mathfrak m^k)_r$. This implies that $(\mathfrak m^k)_r=xR$. It follows that $(\mathfrak m^{2k})_r=(\mathfrak m^k)_r$, hence $x^2R=xR$. Consequently, $x\in I\cap R^{\times}$, and thus $I=R$, a contradiction.

\smallskip
2. Let $(R:\widehat{R})\not=\{0\}$ or let $r$ be finitary. We start with the following assertion.

\begin{enumerate}
\item[{\bf A.}] $(\widehat{R}:A)=(\widehat{R}:A_r)$ for every ideal $A$ of $R$.
\end{enumerate}

{\it Proof of} \ {\bf A.} Let $A$ be an ideal of $R$. It is sufficient to show that $(\widehat{R}:A)\subset (\widehat{R}:A_r)$, since the other inclusion is trivially satisfied. Let $x\in (\widehat{R}:A)$. Then $xA\subset\widehat{R}$. Recall that the $v$ system is the coarsest ideal system on $R$ and the $t$ system is the coarsest finitary ideal system on $R$.

\smallskip
\noindent
CASE 1: \, $(R:\widehat{R})\not=\{0\}$.

We have $\widehat{R}$ is a fractional divisorial ideal of $R$, and thus $xA_r \subset x A_v = (xA)_v \subset (\widehat{R})_v=\widehat{R}$.

\smallskip
\noindent
CASE 2: \, $r$ is finitary.

It follows that $xA_r \subset x A_t = (xA)_t\subset (\widehat{R})_t=\widehat{R}$.\qed ({\bf A})

\bigskip
Note that $A_{v_{\widehat{R}}}=A\widehat{R}$ for every ideal $A$ of $R$, since $\widehat{R}$ is a principal ideal domain. Now let $I,J\in\mathcal{I}_r(R)$ be such that $(IJ)_r=I$. We infer by {\bf A} that $I\widehat{R}J\widehat{R}=(IJ)_{v_{\widehat{R}}}=((IJ)_r)_{v_{\widehat{R}}}=I_{v_{\widehat{R}}}=I\widehat{R}$. Consequently, $J\widehat{R}=\widehat{R}$, since $\widehat{R}$ is a principal ideal domain. Note that $J$ is additively and multiplicatively closed and $J\not\subset\mathfrak m$ for every $\mathfrak m\in\max(\widehat{R})$. Therefore, $J\not\subset\bigcup_{\mathfrak m\in\max(\widehat{R})} \mathfrak m=\widehat{R}\setminus\widehat{R}^{\times}$ by prime avoidance, since $\widehat{R}$ is semilocal. Since $R$ is archimedean, we have $\emptyset\not=J\cap\widehat{R}^{\times}=J\cap R\cap\widehat{R}^{\times}=J\cap R^{\times}$, and thus $J=R$.

\smallskip
3. Let $I,J\in\mathcal{I}_r(R)$ be such that $(IJ)_r=I$. Set $S=(I:I)$. It follows by \cite[Theorem 4.1(b)]{Re12a} and its proof that $S$ is a Mori domain and $JS=S$. Observe that $S$ is also a $\G$-domain, and thus ${\rm spec}(S)$ is finite. Since $J$ is an ideal of $R$ and $J\not\subset\mathfrak m$ for all $\mathfrak m\in\max(S)$ we have $J\not\subset\bigcup_{\mathfrak m\in\max(S)} \mathfrak m=S\setminus S^{\times}$ by prime avoidance. Therefore, $J\cap S^{\times}\not=\emptyset$. Clearly, $R$ is archimedean (since $R$ is a Mori domain). Since $S\subset\widehat{R}$, we have $R^{\times}\subset S^{\times}\cap R\subset\widehat{R}^{\times}\cap R=R^{\times}$, hence $S^{\times}\cap R=R^{\times}$. We infer that $\emptyset\not=J\cap S^{\times}=J\cap R\cap S^{\times}=J\cap R^{\times}$, and thus $J=R$.
\end{proof}

\begin{lemma}\label{4.3}
Let $R$ be a domain and $r$ a finitary ideal system on $R$.
\begin{enumerate}
\item If $R_\mathfrak m$ satisfies the $r_\mathfrak m$-Krull Intersection Theorem for every $\mathfrak m\in r$-$\max(R)$, then $R$ satisfies the $r$-Krull Intersection Theorem.

\smallskip
\item If $\mathcal{I}_{r_\mathfrak m}(R_\mathfrak m)$ is unit-cancellative for every $\mathfrak m\in r$-$\max(R)$, then $\mathcal{I}_r(R)$ is unit-cancellative.
\end{enumerate}
\end{lemma}

\begin{proof}
1. Let $I\in\mathcal{I}_r(R)$ be such that $I\not=R$. Since $r$ is finitary, there is some $\mathfrak m\in r$-$\max(R)$ such that $I\subset \mathfrak m$. We have $I_\mathfrak m\in\mathcal{I}_{r_\mathfrak m}(R_\mathfrak m)$ and $I_\mathfrak m\subset \mathfrak m_\mathfrak m\subsetneq R_\mathfrak m$. Therefore, $\bigcap_{n\in\mathbb{N}_0} (I^n)_r\subset\bigcap_{n\in\mathbb{N}_0} (I_\mathfrak m^n)_{r_\mathfrak m}=\{0\}$.

2. Let $I,J\in\mathcal{I}_r(R)$ be such that $(IJ)_r=I$. Then $(I_\mathfrak mJ_\mathfrak m)_{r_\mathfrak m}=((IJ)_r)_\mathfrak m=I_\mathfrak m$ for all $\mathfrak m\in r$-$\max(R)$, and thus $J_\mathfrak m=R_\mathfrak m$ for all $\mathfrak m\in r$-$\max(R)$. Consequently, $J=R$.
\end{proof}

\begin{corollary}\label{4.4}
Let $R$ be a domain and $r$ a finitary ideal system on $R$.
\begin{enumerate}
\item If $R$ is a Mori domain, then $\mathcal{I}_v(R)$ is unit-cancellative if and only if $\mathcal{I}_v(R)$ has no nontrivial idempotents.

\smallskip
\item Let $R$ be a Mori domain or let $\widehat{R}$ be a Krull domain. If $r$-$\max(R)=\mathfrak{X}(R)$, then $\mathcal{I}_r(R)$ is unit-cancellative.
\end{enumerate}
\end{corollary}

\begin{proof}
1. Let $R$ be a Mori domain and $I\in\mathcal{I}_v(R)$. Note that $R\in\{J\in\mathcal{I}_v(R)\mid (IJ)_v=I\}\subset\{J\in\mathcal{I}_v(R)\mid I\subset J\}$. Since $R$ is a Mori domain we infer that $\{J\in\mathcal{I}_v(R)\mid (IJ)_v=I\}$ has a minimal element with respect to inclusion. Now the assertion follows from Lemma \ref{4.1}.2.

\smallskip
2. Let $r$-$\max(R)=\mathfrak{X}(R)$. By Lemma \ref{4.3}.2 it is sufficient to show that $\mathcal{I}_{r_\mathfrak m}(R_\mathfrak m)$ is unit-cancellative for all $\mathfrak m\in\mathfrak{X}(R)$. Let $\mathfrak m\in\mathfrak{X}(R)$.

CASE 1: \, $R$ is a Mori domain. Then $R_\mathfrak m$ is a one-dimensional local Mori domain.

It follows that $R_\mathfrak m$ is a strongly primary domain, hence $\mathcal{I}_{r_\mathfrak m}(R_\mathfrak m)$ is unit-cancellative by Lemma \ref{4.1}.1 and Proposition \ref{4.2}.1.

CASE 2: \, $\widehat{R}$ is a Krull domain.

Clearly, $R_\mathfrak m$ is a one-dimensional local domain, and thus $R_\mathfrak m$ is an archimedean $\G$-domain. Moreover, we have $\widehat{R_\mathfrak m}=\widehat{R}_\mathfrak m$ is a Krull domain and a $\G$-domain. Therefore, $\widehat{R_\mathfrak m}$ is a semilocal principal ideal domain. (A Krull domain that is a $\G$-domain has finite prime spectrum. Therefore, it is semilocal and it is at most one-dimensional by prime avoidance (since a Krull domain satisfies Krull's principal ideal theorem). A Krull domain that is at most one-dimensional is a Dedekind domain and a semilocal Dedekind domain is a principal ideal domain.) It is clear that $r_\mathfrak m$ is a finitary ideal system on $R_\mathfrak m$. We infer by Proposition \ref{4.2}.2 that $\mathcal{I}_{r_\mathfrak m}(R_\mathfrak m)$ is unit-cancellative.
\end{proof}

\smallskip
A domain is called a {\it Cohen-Kaplansky domain} if it is atomic and has only finitely many atoms up to associates. Our next result is well-known for usual ring ideals (see \cite[Theorem 4.3]{An-Mo92} and \cite{An-Ch08a}).

\begin{proposition}\label{4.5}
Let $R$ be a domain and $r$ an ideal system on $R$. Then the following statements are equivalent{\rm \,:}
\begin{itemize}
\item[(a)] $R$ is a Cohen-Kaplansky domain.

\item[(b)] $\mathcal{I}_r(R)$ is a finitely generated semigroup.

\item[(c)] $\mathcal{I}_r^*(R)$ is a finitely generated semigroup.
\end{itemize}
\end{proposition}

\begin{proof}
First let $R$ be a Cohen-Kaplansky domain. We know that $\mathcal{I}(R)$ is a finitely generated semigroup by \cite[Theorem 4.3]{An-Mo92}, and thus there is some finite $E\subset\mathcal{I}(R)$ such that $\mathcal{I}(R)=[E]$. Set $F=\{J_r\mid J\in E\}$. Then $F$ is finite and it suffices to show that $\mathcal{I}_r(R)=[F]$. It is clear that $[F]\subset\mathcal{I}_r(R)$. Let $I\in\mathcal{I}_r(R)$. Then $I\in\mathcal{I}(R)$, and thus there is some $(\alpha_J)_{J\in E}\in\mathbb{N}_0^E$ such that $I=\prod_{J\in E} J^{\alpha_J}$. We infer that $I=I_r=(\prod_{J\in E} (J_r)^{\alpha_J})_r\in [F]$.

Next let $\mathcal{I}_r(R)$ be a finitely generated semigroup. Then there is some finite $E\subset\mathcal{I}_r(R)$ such that $\mathcal{I}_r(R)=[E]$. Set $F=E\cap\mathcal{I}_r^*(R)$. Then $F$ is finite and it suffices to show that $\mathcal{I}_r^*(R)=[F]$. Obviously, $[F]\subset\mathcal{I}_r^*(R)$. Let $I\in\mathcal{I}_r^*(R)$. Then $I\in\mathcal{I}_r(R)$ and hence there is some $(\alpha_J)_{J\in E}\in\mathbb{N}_0^E$ such that $I=(\prod_{J\in E} J^{\alpha_J})_r$. Recall that if $A\in\mathcal{I}_r^*(R)$ and $B,C\in\mathcal{I}_r(R)$ are such that $A=(BC)_r$, then $B,C\in\mathcal{I}_r^*(R)$. This implies that $\alpha_J=0$ or $J\in F$ for all $J\in E$. Therefore, $I=(\prod_{J\in F} J^{\alpha_J})_r\in [F]$.

Finally let $\mathcal{I}_r^*(R)$ be a finitely generated semigroup. Note that $\mathcal{I}_r^*(R)$ is a cancellative monoid and $\{xR\mid x\in R^{\bullet}\}$ is a saturated submonoid of $\mathcal{I}_r^*(R)$. We infer that $\{xR\mid x\in R^{\bullet}\}$ is a finitely generated monoid, and thus the monoid of principal ideals of $R$ is finitely generated. Consequently, $R$ is a Cohen-Kaplansky domain by \cite[Theorem 4.3]{An-Mo92}.
\end{proof}

Let $R$ be a domain, $r$ an ideal system on $R$ and $I,J\in\mathcal{I}_r(R)$. We say that $I$ and $J$ are {\it $r$-coprime} if $R$ is the only $r$-ideal containing $I$ and $J$.

\begin{lemma}\label{4.6}
Let $R$ be a domain and $r$ a finitary ideal system on $R$. Let $I,J,L\in\mathcal{I}_r(R)$, $n\in\mathbb{N}_0$ and $(I_i)_{i=1}^n$ a finite sequence of elements of $\mathcal{I}_r(R)$.
\begin{enumerate}
\item $I$ and $J$ are $r$-coprime if and only if $\sqrt{I}$ and $\sqrt{J}$ are $r$-coprime. If these equivalent condition are satisfied, then $(IJ)_r=I\cap J$.

\item If $I$ and $J$ are $r$-coprime and $I$ and $L$ are $r$-coprime, then $I$ and $J\cap L$ are $r$-coprime.

\item If for all distinct $i,j\in [1,n]$, $I_i$ and $I_j$ are $r$-coprime, then $(\prod_{i=1}^n I_i)_r=\bigcap_{i=1}^n I_i$.
\end{enumerate}
\end{lemma}

\begin{proof}
1. Since $r$ is finitary, every proper $r$-ideal of $R$ is contained in an $r$-maximal $r$-ideal. Moreover, the radical of every $r$-ideal is again an $r$-ideal. Therefore, $I$ and $J$ are not $r$-coprime if and only if there is some $\mathfrak m\in r$-$\max(R)$ such that $I,J\subset M$ if and only if there is some $\mathfrak m\in r$-$\max(R)$ such that $\sqrt{I},\sqrt{J}\subset\mathfrak m$ if and only if $\sqrt{I}$ and $\sqrt{J}$ are not $r$-coprime.

Now let $I$ and $J$ be $r$-coprime. Then $(I\cup J)_r=R$. Clearly, $(IJ)_r\subset I\cap J$. Let $x\in I\cap J$. Then $x\in xR=x(I\cup J)_r=(xI\cup xJ)_r\subset (IJ)_r$.

2. Let $I$ and $J\cap L$ be not $r$-coprime. Then there is some $\mathfrak m\in r$-$\max(R)$ such that $I,J\cap L\subset\mathfrak m$. Therefore, $J\subset\mathfrak m$ or $L\subset\mathfrak m$. Without restriction let $J\subset\mathfrak m$. Then $I$ and $J$ are not $r$-coprime.

3. This follows by induction from 1 and 2.
\end{proof}

\begin{proposition}\label{4.7}
Let $R$ be a domain and $r$ an $r$-Noetherian ideal system on $R$ with $r$-$\max(R)=\mathfrak{X}(R)$. Then $\mathcal{I}_r(R)$ is an atomic monoid and $\sqrt{J}\in\mathfrak{X}(R)$ for every $J\in\mathcal{A}(\mathcal{I}_r(R))$.
\end{proposition}

\begin{proof}
First we show that $\mathcal{I}_r(R)$ is an atomic monoid. Since $R$ is $r$-Noetherian, we infer that $R$ is a Mori domain and $r$ is a finitary ideal system. Therefore, $\mathcal{I}_r(R)$ is unit-cancellative by Corollary \ref{4.4}.2.

Next we show that $\mathcal{I}_r(R)$ satisfies the ACCP. Let $(J_i)_{i\in\mathbb{N}}$ be an ascending chain of principal ideals of $\mathcal{I}_r(R)$. There is some sequence $(I_i)_{i\in\mathbb{N}}$ of elements of $\mathcal{I}_r(R)$ such that $J_i=\{(I_iA)_r\mid A\in\mathcal{I}_r(R)\}$ for each $i\in\mathbb{N}$. Let $j\in\mathbb{N}$. Then $I_j=(I_jR)_r\in J_j\subset J_{j+1}$, and thus there is some $B\in\mathcal{I}_r(R)$ such that $I_j=(I_{j+1}B)_r$. We infer that $I_i\subset I_{i+1}$ for every $i\in\mathbb{N}$. Since $R$ is $r$-Noetherian, there is some $m\in\mathbb{N}$ such that $I_n=I_m$ for all $n\in\mathbb{N}_{\geq m}$. This clearly implies that $J_n=J_m$ for all $n\in\mathbb{N}_{\geq m}$. Therefore, $\mathcal{I}_r(R)$ satisfies the ACCP. It follows from \cite[Lemma 3.1(1)]{F-G-K-T17} that $\mathcal{I}_r(R)$ is atomic.

Let $I\in\mathcal{I}_r(R)$. We show that $I=(\prod_{\mathfrak p\in\mathfrak{X}(R),I\subset\mathfrak p} (I_\mathfrak p\cap R))_r$ and $\sqrt{I_\mathfrak q\cap R}=\mathfrak q$ for every $\mathfrak q\in\mathfrak{X}(R)$ with $I\subset\mathfrak q$. Clearly, $R$ is a weakly Krull domain, and thus $\{\mathfrak q\in\mathfrak{X}(R)\mid I\subset\mathfrak q\}$ is finite. Moreover, $I=\bigcap_{\mathfrak p\in\mathfrak{X}(R)} I_\mathfrak p=\bigcap_{\mathfrak p\in\mathfrak{X}(R),I\subset\mathfrak p} I_\mathfrak p=\bigcap_{\mathfrak p\in\mathfrak{X}(R),I\subset\mathfrak p} (I_\mathfrak p\cap R)$. Let $\mathfrak q\in\mathfrak{X}(R)$ with $I\subset\mathfrak q$. Since $r$ is finitary it is straightforward to prove that $I_\mathfrak q\cap R$ is an $r$-ideal of $R$. Moreover, $\sqrt{I_\mathfrak q}=\mathfrak q_\mathfrak q$ (since $R_\mathfrak q$ is a one-dimensional local domain), and hence $\sqrt{I_\mathfrak q\cap R}=\mathfrak q$. Therefore, $\bigcap_{\mathfrak p\in\mathfrak{X}(R),I\subset\mathfrak p} (I_\mathfrak p\cap R)$ is a finite intersection of pairwise $r$-coprime $r$-ideals of $R$ by Lemma \ref{4.6}.1 (since their radicals are pairwise $r$-coprime $r$-ideals of $R$). Consequently, $I=(\prod_{\mathfrak p\in\mathfrak{X}(R),I\subset\mathfrak p} (I_\mathfrak p\cap R))_r$ by Lemma \ref{4.6}.3.

Let $J\in\mathcal{A}(\mathcal{I}_{r}(R))$. Then $J=(\prod_{\mathfrak p\in\mathfrak{X}(R),J\subset\mathfrak p} (J_\mathfrak p\cap R))_r$, hence $J=J_\mathfrak p\cap R$ for some $\mathfrak p\in\mathfrak{X}(R)$ such that $J\subset\mathfrak p$. We infer that $\sqrt{J}=\sqrt{J_\mathfrak p\cap R}=\mathfrak p\in\mathfrak{X}(R)$.
\end{proof}

\medskip
\begin{theorem}\label{4.8}
Let $R$ be a domain with $(R:\widehat{R})\not=\{0\}$ and $r$ an $r$-Noetherian ideal system on $R$ with $r$-$\max(R)=\mathfrak{X}(R)$. We set $\mathcal{P}=\{\mathfrak p\in\mathfrak{X}(R)\mid\mathfrak p\not\supset (R:\widehat{R})\}$, $\mathcal{P}^*=\mathfrak{X}(R)\setminus\mathcal{P}$, and let $T$ be the subsemigroup of $\mathcal{I}_r(R)$ generated by $\{I\in\mathcal{A}(\mathcal{I}_r(R))\mid\sqrt{I}\in\mathcal{P}^*\}$.
\begin{enumerate}
\item $T$ is a monoid and $\mathcal{I}_r(R)\cong\mathcal{F}(\mathcal{P})\times T$.

\smallskip
\item $T$ is finitely generated if and only if $\{I\in\mathcal{A}(\mathcal{I}_r(R))\mid\sqrt{I}\in\mathcal{P}^*\}$ is finite.
\end{enumerate}
\end{theorem}

\begin{proof}
We start with the following assertion.
\begin{enumerate}
\item[{\bf A.}\,] For every $\mathfrak p\in\mathcal{P}$, $\mathfrak p$ is $r$-invertible and $\{I\in\mathcal{A}(\mathcal{I}_r(R))\mid\sqrt{I}=\mathfrak p\}=\{\mathfrak p\}$.
\end{enumerate}

{\it Proof of} \ {\bf A.} Let $\mathfrak p\in\mathcal{P}$. It follows from \cite[Theorem 2.10.9]{Ge-HK06a} that $R_\mathfrak p$ is a discrete valuation domain. This implies that $\mathfrak p_\mathfrak q$ is a principal ideal of $R_\mathfrak q$ for every $\mathfrak q\in r$-$\max(R)$, and thus $\mathfrak p$ is an $r$-invertible $r$-ideal of $R$ by \cite[12.3 Theorem]{HK98} (since $R$ is $r$-Noetherian). Clearly, $\mathfrak p\in\{I\in\mathcal{A}(\mathcal{I}_r(R))\mid\sqrt{I}=\mathfrak p\}$. Let $I\in\mathcal{A}(\mathcal{I}_r(R))$ be such that $\sqrt{I}=\mathfrak p$. It is sufficient to show that $I=(\mathfrak p^k)_r$ for some $k\in\mathbb{N}$ (then $I=\mathfrak p$, since $I\in\mathcal{A}(\mathcal{I}_r(R))$). There is some $\ell\in\mathbb{N}$ such that $(\mathfrak p^{\ell})_r\subset I$ (note that $\mathfrak p=J_r$ for some finite $J\subset\mathfrak p$). Since $\mathfrak p$ is $r$-invertible there is some greatest $k\in\mathbb{N}$ such that $I\subset (\mathfrak p^k)_r$ (otherwise we have $(\mathfrak p^{\ell})_r\subset I\subset (\mathfrak p^{\ell+1})_r$, which contradicts the fact that $\mathfrak p$ is a proper $r$-invertible $r$-ideal). Suppose that $I\subsetneq (\mathfrak p^k)_r$. Then $(I\mathfrak p^{-k})_r\subsetneq R$, and thus there is some $\mathfrak q\in\mathfrak{X}(R)$ such that $(I\mathfrak p^{-k})_r\subset\mathfrak q$. We have $I\subset (\mathfrak p^k\mathfrak q)_r\subset\mathfrak q$, hence $\mathfrak p=\mathfrak q$ and $I\subset (\mathfrak p^{k+1})_r$, a contradiction. \qed ({\bf A})

\smallskip
1. By definition, $T$ is a semigroup with identity $R$ and since $\mathcal{I}_r(R)$ is a reduced monoid by Proposition \ref{4.7}, it follows that $T$ is unit-cancellative, and thus $T$ is a monoid.

Let $f:\mathcal{F}(\mathcal{P})\times T\rightarrow\mathcal{I}_r(R)$ be defined by $f((\prod_{\mathfrak p\in\mathcal{P}}\mathfrak p^{\alpha_\mathfrak p},A))=(\prod_{\mathfrak p\in\mathcal{P}}\mathfrak p^{\alpha_\mathfrak p}A)_r$ where $(\alpha_\mathfrak p)_{\mathfrak p\in\mathcal{P}}$ is a formally infinite sequence of nonnegative integers. Clearly, $f$ is a well-defined monoid homomorphism. It remains to show that $f$ is bijective.

First we show that $f$ is injective. Let $(\alpha_\mathfrak p)_{\mathfrak p\in\mathcal{P}}$ and $(\beta_\mathfrak p)_{\mathfrak p\in\mathcal{P}}$ be formally infinite sequences of nonnegative integers and $A,B\in T$ be such that $(\prod_{\mathfrak p\in\mathcal{P}}\mathfrak p^{\alpha_\mathfrak p}A)_r=(\prod_{\mathfrak p\in\mathcal{P}}\mathfrak p^{\beta_\mathfrak p}B)_r$. Since every $\mathfrak q\in\mathcal{P}$ is $r$-invertible it is sufficient to show that $\alpha_\mathfrak q=\beta_\mathfrak q$ for all $\mathfrak q\in\mathcal{P}$ (then it clearly follows that $A=B$). Let $\mathfrak q\in\mathcal{P}$. Suppose that $\alpha_\mathfrak q\not=\beta_\mathfrak q$. Without restriction let $\alpha_\mathfrak q<\beta_\mathfrak q$. Since $\mathfrak q$ is $r$-invertible we have $(\prod_{\mathfrak p\in\mathcal{P},\mathfrak p\not=\mathfrak q}\mathfrak p^{\alpha_\mathfrak p}A)_r=(\prod_{\mathfrak p\in\mathcal{P},\mathfrak p\not=\mathfrak q}\mathfrak p^{\beta_\mathfrak p}\mathfrak q^{\beta_\mathfrak q-\alpha_\mathfrak q}B)_r\subset\mathfrak q$. Consequently, $A\subset\mathfrak q$. Since $A$ is a finite $r$-product of elements of $\{I\in\mathcal{A}(\mathcal{I}_r(R))\mid\sqrt{I}\in\mathcal{P}^*\}$, there is some $B\in\mathcal{A}(\mathcal{I}_r(R))$ such that $\sqrt{B}\in\mathcal{P}^*$ and $B\subset\mathfrak q$. We infer that $\mathfrak q=\sqrt{B}\in\mathcal{P}^*$, a contradiction.

Next we show that $f$ is surjective. Let $I\in\mathcal{I}_r(R)$. Since $\mathcal{I}_r(R)$ is atomic (by Proposition \ref{4.7}), we have $I=(\prod_{i=1}^n I_i)_r$ for some $n\in\mathbb{N}_0$ and some $I_i\in\mathcal{A}(\mathcal{I}_r(R))$. Set $A=(\prod_{i\in [1,n],\sqrt{I_i}\in\mathcal{P}} I_i)_r$ and $B=(\prod_{i\in [1,n],\sqrt{I_i}\in\mathcal{P}^*} I_i)_r$. Then $B\in T$. It follows by Proposition \ref{4.7} that $I=(AB)_r$. It follows by ${\bf A}$ that $A=(\prod_{\mathfrak p\in\mathcal{P}}\mathfrak p^{\alpha_\mathfrak p})_r$ for some formally infinite sequence of nonnegative integers $(\alpha_\mathfrak p)_{\mathfrak p\in\mathcal{P}}$. Consequently, $f((\prod_{\mathfrak p\in\mathcal{P}}\mathfrak p^{\alpha_\mathfrak p},B))=I$.

2. It follows by 1 that $\mathcal{A}(\mathcal{I}_r(R))\cap T=\mathcal{A}(T)$ and $\{I\in\mathcal{A}(\mathcal{I}_r(R))\mid\sqrt{I}\in\mathcal{P}^*\}=\mathcal{A}(T)$. This immediately implies that $T$ is finitely generated if and only if $\{I\in\mathcal{A}(\mathcal{I}_r(R))\mid\sqrt{I}\in\mathcal{P}^*\}$ is finite.
\end{proof}

If $R$ is an order in a quadratic number field and the $r$-ideals are the usual ring ideals, then there is a very explicit number theoretic characterization of when the set $\{I\in\mathcal{A}(\mathcal{I}_{ }(R))\mid\sqrt{I}\in\mathcal{P}^*\}$ is finite (\cite[Corollary 3.8]{F-G-K-T17}). Here we continue with a discussion of generalized Cohen-Kaplansky domains which provide a further setting where the set $\{I\in\mathcal{A}(\mathcal{I}_{ }(R))\mid\sqrt{I}\in\mathcal{P}^*\}$ is finite (see Theorem \ref{4.13}).

A domain is called a {\it generalized Cohen-Kaplansky domain} if it is atomic and has only finitely many atoms up to associates that are not prime elements. Generalized Cohen-Kaplansky domains were introduced by D.D. Anderson, D.F. Anderson, and M. Zafrullah who proved the following characterization (\cite[Corollary 5 and Theorem 6]{An-An-Za92b}).

\begin{lemma}\label{4.9}
For an integral domain $R$ the following statements are equivalent{\rm \,:}
\begin{itemize}
\item[(a)] $R$ is a generalized Cohen-Kaplansky domain.

\smallskip
\item[(b)] $\overline{R}$ is factorial, $R\subset\overline{R}$ is a root extension (i.e., for every $x\in\overline{R}$ there is some $k\in\mathbb{N}$ such that $x^k\in R$), $(R:\overline{R})$ is a principal ideal of $\overline{R}$, and $\overline{R}/(R:\overline{R})$ is finite.
\end{itemize}
If these conditions hold, then $R$ is weakly factorial whence in particular a weakly Krull domain.
\end{lemma}

\begin{lemma}\label{4.10}
Let $R$ be an archimedean domain such that $\widehat{R}$ is factorial, and $\widehat{R}/(R:\widehat{R})$ is finite. Then $R$ is a $\C$-domain.
\end{lemma}

\begin{proof} Observe that $\{x^{-1}R\cap\widehat{R}\mid x\in\widehat{R}^{\bullet}\}\subset\{I\mid I$ is an $R$-submodule of $\widehat{R}$ with $(R:\widehat{R})\subset I\}$. Therefore, $\{x^{-1}R^{\bullet}\cap\widehat{R}^{\bullet}\mid x\in\widehat{R}^{\bullet}\}$ is finite (since $\widehat{R}/(R:\widehat{R})$ is finite), and thus $\mathcal{C}(R^{\bullet},\widehat{R}^{\bullet})$ is finite by \cite[Lemma 2.8.2.2]{Ge-HK06a}. Since $\widehat{R}^{\times}\cap R=R^{\times}$, we infer that $R^{\bullet}$ is a $\C$-monoid.
\end{proof}

\begin{proposition}\label{4.11}
Let $R$ be a generalized Cohen-Kaplansky domain and $r$ a finitary ideal system on $R$.
\begin{enumerate}
\item $R$ is a $\C$-domain and a Mori domain, and $v$-$\max(R)=\mathfrak{X}(R)$.

\smallskip
\item If $T\subset R^{\bullet}$ a multiplicatively closed subset, then $T^{-1}R$ is a generalized Cohen-Kaplansky domain.

\smallskip
\item $R$ is a Cohen-Kaplansky domain if and only if $R$ is a $\G$-domain. In particular, if $\mathfrak p\in\mathfrak{X}(R)$, then $R_\mathfrak p$ is a Cohen-Kaplansky domain.

\smallskip
\item If $r$-$\max(R)=\mathfrak{X}(R)$, then $R$ is $r$-Noetherian.
\end{enumerate}
\end{proposition}

\begin{proof}
1. Clearly, $\overline{R}$ is factorial and $\overline{R}/(R:\overline{R})$ is finite by Lemma \ref{4.9}. Consequently, $\overline{R}=\widehat{R}$, $(R:\widehat{R})\not=\{0\}$, and $\widehat{R}^{\times}\cap R=\overline{R}^{\times}\cap R=R^{\times}$. We infer by Lemma \ref{4.10} that $R$ is a $\C$-domain, and hence $R$ is a Mori domain by \cite[Theorem 2.9.13]{Ge-HK06a}. It follows from Lemma \ref{4.9} that $R$ is a weakly Krull domain. Since $R$ is a Mori domain, we infer that $v$-$\max(R)=\mathfrak{X}(R)$ (note that $v$-$\max(R)=\emptyset=\mathfrak{X}(R)$ if $R$ is a field).

2. By Lemma \ref{4.9}, we know that $\overline{R}$ is factorial, $R\subset\overline{R}$ is a root extension, $(R:\overline{R})$ is a principal ideal of $\overline{R}$ and $\overline{R}/(R:\overline{R})$ is finite. Again by Lemma \ref{4.9}, it is sufficient to show that $\overline{T^{-1}R}$ is factorial, $T^{-1}R\subset\overline{T^{-1}R}$ is a root extension, $(T^{-1}R:\overline{T^{-1}R})$ is a principal ideal of $\overline{T^{-1}R}$ and $\overline{T^{-1}R}/(T^{-1}R:\overline{T^{-1}R})$ is finite.

It is clear that $\overline{T^{-1}R}=T^{-1}\overline{R}$ is factorial. Let $x\in\overline{T^{-1}R}=T^{-1}\overline{R}$. Then $tx\in\overline{R}$ for some $t\in T$, and hence $(tx)^k\in R$ for some $k\in\mathbb{N}$. This implies that $x^k\in T^{-1}R$. Therefore, $T^{-1}R\subset\overline{T^{-1}R}$ is a root extension. We have $R$ is a Mori domain by 1, $(R:\overline{R})\not=\{0\}$ and $\overline{R}=\widehat{R}$ is a fractional divisorial ideal of $R$. Therefore, $T^{-1}(R:\overline{R})=(T^{-1}R:T^{-1}\overline{R})=(T^{-1}R:\overline{T^{-1}R})$ by \cite[Proposition 2.2.8.1]{Ge-HK06a}. Because of this it is clear that $(T^{-1}R:\overline{T^{-1}R})$ is a principal ideal of $\overline{T^{-1}R}$.

Finally, we have $\overline{T^{-1}R}/(T^{-1}R:\overline{T^{-1}R})=T^{-1}\overline{R}/T^{-1}(R:\overline{R})\cong (T+(R:\overline{R})/(R:\overline{R}))^{-1}(\overline{R}/(R:\overline{R}))$ is finite (since both $T+(R:\overline{R})/(R:\overline{R})$ and $\overline{R}/(R:\overline{R})$ are finite).

3. If $R$ is a Cohen-Kaplansky domain, then $R$ is semilocal and $\dim(R)\leq 1$ by \cite[Theorem 4.3]{An-Mo92}, and thus $R$ is a $\G$-domain. Now let $R$ be a $\G$-domain. Since $R$ is also a Mori domain by 1, $R$ has only finitely many prime ideals. Therefore, $R$ has only finitely many prime elements up to associates, hence $R$ has only finitely many atoms up to associates. We infer that $R$ is a Cohen-Kaplansky domain.

Now let $\mathfrak p\in\mathfrak{X}(R)$. Then $R_\mathfrak p$ is a $\G$-domain and it is a generalized Cohen-Kaplansky domain by 2. Therefore, $R_\mathfrak p$ is a Cohen-Kaplansky domain.

4. Let $r$-$\max(R)=\mathfrak{X}(R)$. If $\mathfrak m\in\mathfrak{X}(R)$, then it follows by 3 that $R_\mathfrak m$ is a Cohen-Kaplansky domain, hence $R_\mathfrak m$ is Noetherian by \cite[Theorem 4.3]{An-Mo92}. Let $(I_i)_{i\in\mathbb{N}}$ be an ascending chain of $r$-ideals of $R$. Without restriction let $I_1\not=\{0\}$. Set $\mathcal{Q}=\{\mathfrak m\in\mathfrak{X}(R)\mid I_1\subset\mathfrak m\}$. Then $\mathcal{Q}$ is finite, and for all $\mathfrak m\in\mathfrak{X}(R)$ with $I_1\not\subset\mathfrak m$ and all $\ell\in\mathbb{N}$ we have $(I_{\ell})_\mathfrak m=(I_1)_\mathfrak m$. Clearly, there is some $(n_\mathfrak m)_{\mathfrak m\in\mathcal{Q}}\in\mathbb{N}^\mathcal Q$ such that for all $m\in\mathcal{Q}$ and $\ell\in\mathbb{N}_{\ge n_\mathfrak m}$ it follows that $(I_{\ell})_\mathfrak m=(I_{n_\mathfrak m})_\mathfrak m$. Set $n=\max(\{n_\mathfrak m\mid\mathfrak m\in\mathcal{Q}\}\cup\{1\})$. Obviously, $(I_{\ell})_\mathfrak m=(I_n)_\mathfrak m$ for all $\mathfrak m\in\mathfrak{X}(R)$ and $\ell\in\mathbb{N}_{\ge n}$. We infer that $I_{\ell}=I_n$ for all $\ell\in\mathbb{N}_{\ge n}$.
\end{proof}

If $R$ is an integral domain, $r$ is an ideal system on $R$, and $\mathfrak m\in r$-$\max(R)$, then let $\mathfrak A_{r,\mathfrak m}(R)$ be the set of all $I\in\mathcal{I}_r(R)$ with $\sqrt{I}=\mathfrak m$ such that for all $J,L\in\mathcal{I}_r(R)$ with $I=(JL)_r$ it follows that $J=R$ or $L=R$. Note that if $\mathcal{I}_r(R)$ is unit-cancellative, then $\mathfrak A_{r,\mathfrak m}(R)$ is the set of atoms of $\mathcal{I}_r(R)$ whose radical is equal to $\mathfrak m$.

\begin{lemma}\label{4.12}
Let $R$ be a domain, $r$ a finitary ideal system on $R$, and $\mathfrak m\in r$-$\max(R)$. Let $\varphi\colon\mathfrak A_{r,\mathfrak m}(R)\rightarrow\mathfrak A_{r_\mathfrak m,\mathfrak m_\mathfrak m}(R_\mathfrak m)$ be defined by $\varphi(I)=I_\mathfrak m$ for all $I\in\mathfrak A_{r,\mathfrak m}(R)$ and let $\psi\colon\mathfrak A_{r_\mathfrak m,\mathfrak m_\mathfrak m}(R_\mathfrak m)\rightarrow\mathfrak A_{r,\mathfrak m}(R)$ be defined by $\psi(J)=J\cap R$ for all $J\in\mathfrak A_{r_\mathfrak m,\mathfrak m_\mathfrak m}(R_\mathfrak m)$. Then $\varphi$ and $\psi$ are mutually inverse bijections.
\end{lemma}

\begin{proof}
Let $f\colon\{I\in\mathcal{I}_r(R)\mid\sqrt{I}=\mathfrak m\}\rightarrow\{J\in\mathcal{I}_{r_\mathfrak m}(R_\mathfrak m)\mid\sqrt{J}=\mathfrak m_\mathfrak m\}$ be defined by $f(I)=I_\mathfrak m$ for each $I\in\mathcal{I}_r(R)$ such that $\sqrt{I}=\mathfrak m$. Let $g\colon\{J\in\mathcal{I}_{r_\mathfrak m}(R_\mathfrak m)\mid\sqrt{J}=\mathfrak m_\mathfrak m\}\rightarrow\{I\in\mathcal{I}_r(R)\mid\sqrt{I}=\mathfrak m\}$ be defined by $f(J)=J\cap R$ for each $J\in\mathcal{I}_{r_\mathfrak m}(R_\mathfrak m)$ such that $\sqrt{J}=\mathfrak m_\mathfrak m$. It is straightforward to prove that $f$ and $g$ are mutually inverse bijections. For instance, note that if $I\in\mathcal{I}_r(R)$ with $\sqrt{I}=\mathfrak m$, then $I$ is $\mathfrak m$-primary, and hence $I_\mathfrak m\cap R=I$. To prove the assertion it is sufficient to show that $f(\mathfrak A_{r,\mathfrak m}(R))=\mathfrak A_{r_\mathfrak m,\mathfrak m_\mathfrak m}(R_\mathfrak m)$.

First let $I\in\mathfrak A_{r,\mathfrak m}(R)$ and $J^{\prime},L^{\prime}\in\mathcal{I}_{r_\mathfrak m}(R_\mathfrak m)$ be such that $I_\mathfrak m=(J^{\prime}L^{\prime})_{r_\mathfrak m}$. Assume to the contrary that $J^{\prime},L^{\prime}\not=R_\mathfrak m$. Then $\sqrt{J^{\prime}}=\sqrt{L^{\prime}}=\mathfrak m_\mathfrak m$. Set $J=J^{\prime}\cap R$ and $L=L^{\prime}\cap R$. Then $J,L\in\mathcal{I}_r(R)$ and $\sqrt{J}=\sqrt{L}=\mathfrak m$, and hence $\sqrt{(JL)_r}=\mathfrak m$. Observe that $I_\mathfrak m=(J^{\prime}L^{\prime})_{r_\mathfrak m}=(J_\mathfrak mL_\mathfrak m)_{r_\mathfrak m}=((JL)_r)_\mathfrak m$. We infer that $I=(JL)_r$, and thus $J=R$ or $L=R$. Therefore, $J^{\prime}=R_\mathfrak m$ or $L^{\prime}=R_\mathfrak m$, a contradiction.

Next let $J\in\mathfrak A_{r_\mathfrak m,\mathfrak m_\mathfrak m}(R_\mathfrak m)$ and set $I=J\cap R$. Let $A,B\in\mathcal{I}_r(R)$ be such that $I=(AB)_r$. Then $J=I_\mathfrak m=(A_\mathfrak mB_\mathfrak m)_{r_\mathfrak m}$, and hence $A_\mathfrak m=R_\mathfrak m$ or $B_\mathfrak m=R_\mathfrak m$. It follows that $A=R$ or $B=R$ (since $\sqrt{A},\sqrt{B}\supset\mathfrak m$).
\end{proof}

\medskip
\begin{theorem}\label{4.13}
Let $R$ be a generalized Cohen-Kaplansky domain, and $r$ an ideal system on $R$. Let $\mathcal{P}=\{\mathfrak p\in\mathfrak{X}(R)\mid\mathfrak p\not\supset (R:\widehat{R})\}$, $\mathcal{P}^*=\mathfrak{X}(R)\setminus\mathcal{P}$, and let $T$ be the subsemigroup of $\mathcal{I}_r(R)$ generated by $\{I\in\mathcal{A}(\mathcal{I}_r(R))\mid\sqrt{I}\in\mathcal{P}^*\}$. Suppose that one of the following conditions is satisfied{\rm \,:}
\begin{itemize}
\item[(a)] $r$ is finitary and $r$-$\max(R)=\mathfrak{X}(R)$.
\item[(b)] $r=v$.
\item[(c)] $R$ is a $\G$-domain.
\end{itemize}
Then $R$ is $r$-Noetherian, $\{I\in\mathcal{A}(\mathcal{I}_r(R))\mid\sqrt{I}\in\mathcal{P}^*\}$ is finite, $T$ is a finitely generated monoid, and $\mathcal{I}_r(R)\cong\mathcal{F}(\mathcal{P})\times T$.
\end{theorem}

\begin{proof}
CASE 1: \, $r$ is finitary and $r$-$\max(R)=\mathfrak{X}(R)$.

By Proposition \ref{4.11}.4 we have $R$ is $r$-Noetherian. Clearly, $(R:\widehat{R})\not=\{0\}$. Consequently, $T$ is a monoid and $\mathcal{I}_r(R)\cong\mathcal{F}(\mathcal{P})\times T$ by Theorem \ref{4.8}.1. Let $\mathfrak m\in\mathcal{P}^*$. Then $R_\mathfrak m$ is a Cohen-Kaplansky domain by Proposition \ref{4.11}.3. We infer by Proposition \ref{4.5} that $\mathcal{I}_{r_\mathfrak m}(R_\mathfrak m)$ is finitely generated, hence $\mathcal{A}(\mathcal{I}_{r_\mathfrak m}(R_\mathfrak m))$ is finite. It follows by Lemma \ref{4.12} that $\{I\in\mathcal{A}(\mathcal{I}_r(R))\mid\sqrt{I}=\mathfrak m\}$ is finite. Since $\mathcal{P}^*$ is finite, we have $\{I\in\mathcal{A}(\mathcal{I}_r(R))\mid\sqrt{I}\in\mathcal{P}^*\}$ is finite, and thus $T$ is finitely generated by Theorem \ref{4.8}.2.

CASE 2: \, $r=v$.

By Proposition \ref{4.11}.1 we have $R$ is a Mori domain and $v$-$\max(R)=\mathfrak{X}(R)$. This implies that $v$ is finitary, and hence the assertion follows from CASE 1.

CASE 3: \, $R$ is a $\G$-domain.

By Proposition \ref{4.11}.3 it follows that $R$ is a Cohen-Kaplansky domain. Without restriction we may suppose that $R$ is not a field. By \cite[Theorem 4.3]{An-Mo92}, $R$ is one-dimensional Noetherian whence $r$-$\max(R)=\mathfrak X(R)$ and $R$ is $r$-Noetherian. Thus we are back to CASE 1.
\end{proof}

\smallskip
We continue with an example of a domain $R$ for which $\mathcal I_v(R)$ is not unit-cancellative (whence the statements of Theorem \ref{4.13} do not hold) but $R$ satisfies three of the four conditions in Lemma \ref{4.9}.(b) characterizing generalized Cohen-Kaplansky domains.

\begin{example}\label{4.14}
There is a two-dimensional Noetherian domain $R$ satisfying the following properties:
\begin{itemize}
\item $\overline{R}$ is local Noetherian and factorial.
\item $R_\mathfrak p$ is a discrete valuation domain for all $\mathfrak p\in\mathfrak{X}(R)$.
\item $(R:\overline{R})$ is a nontrivial idempotent of $\mathcal{I}_v(R)$ and $\mathcal{I}_v(R)$ is not unit-cancellative.
\item $\overline{R}/(R:\overline{R})$ is finite and $R\subset\overline{R}$ is a root extension.
\end{itemize}
\end{example}

\begin{proof}
Let $S$ be a discrete valuation domain, and $d\in S^{\bullet}\setminus S^{\times}$ such that $S/dS$ is finite. Let $X$ be an indeterminate over $S$ and $R=\{f\in S[\![X]\!]\mid $ the linear coefficient of $f$ is divisible by $d\}$. As shown in \cite[Example 5.5.1]{Re12a}, $R$ is a two-dimensional Noetherian domain, $\overline{R}$ is local Noetherian and factorial, $R_\mathfrak p$ is a discrete valuation domain for all $\mathfrak p\in\mathfrak{X}(R)$, and $(R:\overline{R})$ is a nontrivial idempotent of $\mathcal{I}_v(R)$. Clearly, $\mathcal{I}_v(R)$ is not unit-cancellative. Along the same lines as in the proof of \cite[Lemma 5.3.3]{Re12a} it follows that $\overline{R}/(R:\overline{R})$ is finite. Set $k=|S/dS|$. It is straightforward to prove that $f^k\in R$ for all $f\in S[\![X]\!]=\overline{R}$, and thus $R\subset\overline{R}$ is a root extension.
\end{proof}

We end this section with our main arithmetical result. It combines the semigroup theoretical work in Section \ref{3} with the ideal theoretic results of the present section.

\medskip
\begin{theorem} \label{4.15}
Let $R$ be a domain with $(R:\widehat{R})\not=\{0\}$, $r$ an $r$-Noetherian ideal system on $R$ with $r$-$\max(R)=\mathfrak{X}(R)$, and suppose that $\{I\in\mathcal{A}(\mathcal{I}_r(R))\mid\sqrt{I}\supset (R\colon\widehat R)\}$ is finite.
\begin{enumerate}
\item $\mathcal{I}_r(R)$ has finite monotone catenary degree and finite successive distance. In particular, the catenary degree and the set of distances are finite.

\smallskip
\item $\mathcal I_r(R)$ satisfies the Structure Theorem for Unions.
\end{enumerate}
\end{theorem}

\begin{proof}
We use all notation as in Theorem \ref{4.8} and that $\mathcal{I}_r(R)\cong\mathcal{F}(\mathcal{P})\times T$ where $T$ is a finitely generated monoid.

1. Since $\mathsf c_{\mon} \big( \mathcal F (\mathcal P) \times T \big) = \mathsf c_{\mon} (T)$ and $\delta \big( \mathcal F (\mathcal P) \times T \big) = \delta (T)$, it suffices to prove the assertion for the monoid $T$.
Theorem \ref{3.1} implies that $\delta (T) < \infty$ and that $\mathsf c_{\mon} (T) < \infty$. Since $1 + \sup \Delta (T) \le \mathsf c (T) \le \mathsf c_{\mon} (T)$, the remaining assertions follow.

\smallskip
2. It can be seen from the definitions that
\[
\Delta \big( \mathcal F (\mathcal P) \times T \big) = \Delta (T) \quad \text{and} \quad \rho_k \big( \mathcal F (\mathcal P) \times T \big) = \rho_k (T)
\]
for every $k \in \N$.
By \cite[Proposition 3.4]{F-G-K-T17}, there exists a constant $M \in \N$ such that $\rho_{k+1}(T) \le \rho_k (T) + M$ for all $k \in \N$ whence $\rho_{k+1} \big( \mathcal F (\mathcal P) \times T \big) \le \rho_k \big( \mathcal F (\mathcal P) \times T \big) + M$ for all $k \in \N$ Since $\Delta (T)$ is finite by 1., it follows that $\Delta \big( \mathcal F (\mathcal P) \times T \big)$ is finite.
Therefore, all assumptions of \cite[Theorem 2.2]{F-G-K-T17} are satisfied whence $\mathcal F (\mathcal P) \times T$ satisfies the Structure Theorem for Unions.
\end{proof}

\medskip
\section{Monoids of $v$-invertible $v$-ideals in weakly Krull monoids} \label{5}
\medskip

Weakly Krull domains were introduced by Anderson, Anderson, Mott, and Zafrullah \cite{An-An-Za92b, An-Mo-Za92}. Halter-Koch gave a divisor-theoretic characterization and showed that a domain is weakly Krull if and only if its multiplicative monoid of non-zero elements is a weakly Krull monoid (\cite{HK95a}).
We will restrict to the setting of $v$-Noetherian monoids and domains and recall that a (commutative cancellative) $v$-Noetherian monoid $H$ is weakly Krull if and only if $v$-$\max (H) = \mathfrak X (H)$ (\cite[Theorem 24.5]{HK98}). Thus one-dimensional Mori domains are weakly Krull, and by Proposition \ref{4.11} generalized Cohen-Kaplansky domains are weakly Krull Mori domains.

In this section we study the monotone catenary degree of the monoid $\mathcal I_v^*(H)$ of $v$-invertible $v$-ideals, where $H$ is a weakly Krull Mori monoid with nonempty conductor $(H \DP \widehat H)$.
This monoid is a direct product of a free abelian part and of finitely many finitely primary monoids (see \eqref{structure}).
The seminormal case has already been studied in detail. Indeed, if $H$ is a seminormal $v$-Noetherian weakly Krull monoid with proper nonempty conductor such that $H_\mathfrak p$ is finitely primary for all $\mathfrak p\in\mathfrak{X}(H)$, then the monotone catenary degree of $\mathcal I_v^*(H)$ is either $2,3$, or $5$, and it is well-understood which case occurs (\cite[Theorem 5.8]{Ge-Ka-Re15a}).

However, in general and even in the local case (thus for finitely primary monoids), the monotone catenary degree may be infinite (Remark \ref{5.3}). We study a special class of finitely primary monoids, called strongly ring-like, which was introduced by Hassler in \cite{Ha09c}. Strongly ring-like monoids of rank at most two (the restriction on the rank is essential, as outlined in Remark \ref{5.3}) have finite monotone catenary degree (Proposition \ref{5.12}) and the same is true for $\mathcal I_v^*(H)$ provided that the localizations $H_{\mathfrak p}$ are strongly ring-like of rank at most two (Theorem \ref{5.13}).

\smallskip
We begin with the local case.
A monoid $H$ is said to be {\it finitely primary} if there are $s, \alpha \in \N$ and a factorial monoid $F = F^{\times} \time \mathcal F ( \{p_1, \ldots, p_s\})$ such that $H \subset F$ with
\begin{equation}\label{eq:basic2}
H \setminus H^{\times} \subset p_1 \cdot \ldots \cdot p_s F \quad \text{and} \quad (p_1 \cdot \ldots \cdot p_s)^{\alpha}F \subset H \,.
\end{equation}
Let $H \subset F$ be finitely primary as above. Then $s$ is called the rank of $H$, $\alpha$ is called an exponent of $H$, and $\mathsf v \colon H \to \N_0^s$, defined by $(a \mapsto \big(\mathsf v_{p_1}(a), \ldots, \mathsf v_{p_s} (a) \big)$ for all $a \in H$, denotes the map from $H$ to its value semigroup $\mathsf v (H)$. It is well-known (\cite[Theorems 2.9.2 and 3.1.5]{Ge-HK06a}) that $H$ is primary, that $F$ is the complete integral closure of $H$, and that $s = |\mathfrak X ( \widehat H)|$. Clearly, every finitely primary monoid is strongly primary.

The following two lemmas gather the main arithmetical properties of finitely primary monoids and the connection to ring theory.

\medskip
\begin{lemma}\label{5.1}
Let $H$ be a finitely primary monoid. Then $H$ has finite catenary degree, finite set of distances, and it satisfies the Structure Theorem for Sets of Lengths. Moreover, if the rank of $H$ is greater than one, then $H$ is not half-factorial.
\end{lemma}

\begin{proof}
By \cite[Corollary 4.5.5]{Ge-HK06a} we have $H$ satisfies the Structure Theorem for Sets of Lengths. It follows from \cite[Theorems 2.9.2.4, 3.1.5.2, 3.1.1.2]{Ge-HK06a} that $\Delta(H)$ is finite.
\end{proof}

\medskip
\begin{lemma}\label{5.2}
Let $R$ be a domain.
\begin{enumerate}
\item $R^{\bullet}$ is finitely primary if and only if $R$ is one-dimensional local, $(R \DP \widehat R) \ne \{0\}$, and $\widehat R$ is a semilocal principal ideal domain.

\item If $R$ is a one-dimensional local Mori domain such that $(R \DP \widehat R) \ne \{0\}$, then $ R^{\bullet}$ is finitely primary of rank $|\mathfrak X (\widehat R)|$.
\end{enumerate}
\end{lemma}

\begin{proof}
See \cite[Proposition 2.10.7]{Ge-HK06a}.
\end{proof}

\medskip
\begin{remark}\label{5.3}~

1. Let $H \subset F$ be a finitely primary monoid as in \eqref{eq:basic2}. By \cite[Corollary 2.9.8]{Ge-HK06a}, $H$ is a C-monoid if and only if the following two conditions are fulfilled:
\begin{enumerate}
\item[(a)]
There exists a subgroup $V \subset \widehat H^\times$ of finite index such that \ $V(H \setminus H^\times) \subset H$.

\smallskip

\item[(b)]
There exists some $\alpha \in \N$ such that, for every $j\in [1,s]$ and $a\in p_j^\alpha \widehat H$, we have $a\in H$ if and only if \ $p_j^\alpha a\in H$.
\end{enumerate}
In general, finitely primary monoids need neither be $v$-Noetherian nor C-monoids (see \cite{HK-Ha-Ka04, Re13a}).

\smallskip
2. Let $H \subset F$ be a finitely primary monoid as above. Then $H_{\red}$ is finitely generated if and only if $s=1$ and $(F^{\times}\DP H^{\times})<\infty$ (\cite[Theorem 2.9.2]{Ge-HK06a}). If $H_{\red}$ is finitely generated, then $\delta_w(H)\le\delta(H)<\infty$ and $\mathsf c_{\mon}(H)<\infty$ by Theorem \ref{3.1}.

\smallskip
3. In contrast to Lemma \ref{5.1}, there is a finitely primary monoid of rank one with $\mathsf c_{\mon} (H) = \delta (H)=\infty$. Furthermore, there are finitely primary monoids of rank two and exponent two having infinite monotone catenary degree (see \cite[Remark 4.6, Examples 4.5 and 4.16]{Fo06a}).

\smallskip
4. There are one-dimensional local Noetherian domains $R$ with maximal ideal $\mathfrak m$ such that:
\begin{itemize}
\item $\overline R$ is a finitely generated $R$-module and $R/\mathfrak m$ is finite.
\item ($\overline R$ has $2$ maximal ideals and $\delta(R^{\bullet})=\infty$) or ($\overline R$ has $3$ maximal ideals and $\mathsf c_{\rm mon}(R^{\bullet})=\infty$).
\end{itemize}
(\cite[Examples 6.3 and 6.5]{Ha09c}). By Lemma \ref{5.2}, $R^{\bullet}$ is finitely primary of rank two or three.
\end{remark}

\smallskip
The examples discussed in Remark \ref{5.3} show that finitely primary monoids need to satisfy further structural properties if we want their monotone catenary degree to be finite. Such properties were introduced by Hassler in \cite{Ha09c}, and we recall the definition.

\medskip
\begin{definition}\label{5.4}
Let $H$ be a finitely primary monoid of rank $s\in\mathbb N$ such that there exist some exponent $\alpha\in\mathbb N$ of $H$ and some system $\{p_i\mid i\in [1,s]\}$ of representatives of the prime elements of $\widehat H$ with the following property: for all $i\in [1,s]$ and for all $a\in\widehat H$ with $\mathsf v_{p_i} (a)\ge\alpha$ we have $p_ia\in H$ if and only if $a\in H$. Then $H$ is said to be
\begin{itemize}
\item {\it ring-like} if ${\widehat H}^{\times}/H^{\times}$ is finite or $\{(\mathsf v_{p_i}(a))_{i=1}^s\mid a\in H\setminus H^{\times}\}\subset\mathbb N^s$ has a smallest element with respect to the partial order.

\smallskip
\item {\it strongly ring-like} if ${\widehat H}^{\times}/H^{\times}$ is finite and $\{(\mathsf v_{p_i}(a))_{i=1}^s\mid a\in H\setminus H^{\times}\}\subset\mathbb N^s$ has a smallest element with respect to the partial order.
\end{itemize}
\end{definition}

Let $H$ be a ring-like monoid of rank $s$ and exponent $\alpha$ and $\{p_i\mid i\in [1,s]\}$ a system of representatives of the prime elements of $\widehat H$. We say that $\alpha$ and $\{p_i\mid i\in [1,s]\}$ are suitably chosen if for all $i\in [1,s]$ and for all $a\in\widehat H$ with $\mathsf v_{p_i} (a)\ge\alpha$ we have $p_ia\in H$ if and only if $a\in H$.

We continue with a ring theoretical analysis which one-dimensional local domains are strongly ring-like (Propositions \ref{5.6} and \ref{5.9}).
The characterization mentioned in Remark \ref{5.3}.1 shows that strongly ring-like monoids are C-monoids and hence in particular they are $v$-Noetherian.

\medskip
\begin{lemma}\label{5.5}
Let $S$ be a commutative ring, $a,b\in S$, and $\mathcal{R}\subset S$ such that $x-y\in S^{\times}$ for all distinct $x,y\in\mathcal{R}$. If $\mathfrak{I}$ is a finite set of ideals of $S$ such that $|\mathfrak{I}|<|\mathcal{R}|$ and $b\not\in I$ for all $I\in\mathfrak{I}$, then there is some $\eta\in\mathcal{R}$ such that $a+\eta b\not\in\bigcup_{I\in\mathfrak{I}} I$.
\end{lemma}

\begin{proof} It is sufficient to show by induction that for all $k\in\mathbb{N}_0$ and all sets of ideals $\mathfrak{I}$ of $S$ such that $k=|\mathfrak{I}|<|\mathcal{R}|$ and $b\not\in I$ for all $I\in\mathfrak{I}$ it follows that there is some $\eta\in\mathcal{R}$ such that $a+\eta b\not\in\bigcup_{I\in\mathfrak{I}} I$.

The assertion is clear for $k=0$. Now let $k\in\mathbb{N}_0$ and $\mathfrak{I}$ a set of ideals of $S$ be such that $k+1=|\mathfrak{I}|<|\mathcal{R}|$ and $b\not\in I$ for all $I\in\mathfrak{I}$. By the induction hypothesis there is some $(\eta_I)_{I\in\mathfrak{I}}\in\mathcal{R}^{\mathfrak{I}}$ such that for all $I\in\mathfrak{I}$ it follows that $a+\eta_I b\not\in\bigcup_{J\in\mathfrak{I}\setminus\{I\}} J$.

\smallskip
\noindent
CASE 1: \, $a+\eta_J b\not\in J$ for some $J\in\mathfrak{I}$.

Then $a+\eta_J b\not\in\bigcup_{I\in\mathfrak{I}} I$.

\smallskip
\noindent
CASE 2: \, $a+\eta_I b\in I$ for all $I\in\mathfrak{I}$.

There is some $\eta\in\mathcal{R}\setminus\{\eta_I\mid I\in\mathfrak{I}\}$. Assume that $a+\eta b\in J$ for some $J\in\mathfrak{I}$. Then $(\eta-\eta_J)b=a+\eta b-(a+\eta_J b)\in J$, hence $b\in J$, a contradiction. Therefore, $a+\eta b\not\in\bigcup_{I\in\mathfrak{I}} I$.
\end{proof}

If $R$ is a domain such that $R^{\bullet}$ is finitely primary and $\mathfrak m=R\setminus R^{\times}$, then set $V(\mathfrak m^{\bullet})=\linebreak\{(\mathsf v_\mathfrak q(a))_{\mathfrak q\in\max(\widehat{R})}\mid a\in\mathfrak m^{\bullet}\}$.

\medskip
\begin{proposition}\label{5.6}
Let $R$ be a domain such that $R^{\bullet}$ is finitely primary and $\mathfrak m=R\setminus R^{\times}$.
\begin{enumerate}
\item[\textnormal{1.}] If $|\max(\widehat{R})|\leq |R/\mathfrak m|$, then $V(\mathfrak m^{\bullet})$ has a smallest element.
\item[\textnormal{2.}] The following statements are equivalent{\rm \,:}
\begin{enumerate}
\item[(a)] $\widehat{R}^{\times}/R^{\times}$ is finite.
\item[(b)] $R$ is a discrete valuation domain or $(R$ is Noetherian and $|R/\mathfrak m|<\infty)$.
\item[(c)] $R$ is a discrete valuation domain or $(\widehat{R}$ is a finitely generated $R$-module and $|R/\mathfrak m|<\infty)$.
\item[(d)] $R$ is an \FF-domain.
\end{enumerate}
\item[\textnormal{3.}] There are some $\alpha,s\in\mathbb{N}$ and some system $\{p_i\mid i\in [1,s]\}$ of representatives of prime elements of $\widehat{R}$ such that for all $i\in [1,s]$ and $a\in\widehat{R}$ with $\mathsf v_{p_i}(a)\geq\alpha$ it follows that $p_ia\in R$ if and only if $a\in R$.
\end{enumerate}
\end{proposition}

\begin{proof}
1. Let $|\max(\widehat{R})|\leq |R/\mathfrak m|$. By \cite[Theorem 1.5.3]{Ge-HK06a} it suffices to show that $|{\rm Min}(V(\mathfrak m^{\bullet}))|\leq 1$. Assume to the contrary that there are distinct $x,y\in {\rm Min}(V(\mathfrak m^{\bullet}))$. There are some $e,f\in\mathfrak m^{\bullet}$ such that $x=(\mathsf v_\mathfrak q(e))_{\mathfrak q\in\max(\widehat{R})}$ and $y=(\mathsf v_\mathfrak q(f))_{\mathfrak q\in\max(\widehat{R})}$. Since $\widehat{R}$ is a principal ideal domain, there is some $d\in {\rm GCD}_{\widehat{R}}(e,f)$. Set $a=\frac{e}{d}$ and $b=\frac{f}{d}$. There is some $\mathcal{R}\subset R^{\times}$ such that $|\mathcal{R}|=|R/\mathfrak m|-1$ and $v-w\in R^{\times}\subset\widehat{R}^{\times}$ for all distinct $v,w\in\mathcal{R}$. Set $\mathfrak{I}=\{\mathfrak q\in\max(\widehat{R})\mid\mathsf v_\mathfrak q(e)=\mathsf v_\mathfrak q(f)\}$. We have $\mathfrak{I}\subset\{\mathfrak q\in\max(\widehat{R})\mid b\not\in\mathfrak q\}$. Since $x,y\in {\rm Min}(V(\mathfrak m^{\bullet}))$ are distinct it follows that $|\mathfrak{I}|<|\max(\widehat{R})|-1\leq |R/\mathfrak m|-1=|\mathcal{R}|$. By Lemma \ref{5.5} there is some $\eta\in\mathcal{R}$ such that $a+\eta b\not\in\bigcup_{\mathfrak q\in\mathfrak{I}} \mathfrak q$.

Next we show that $\mathsf v_\mathfrak n(e+\eta f)=\min \{ \mathsf v_\mathfrak n(e),\mathsf v_\mathfrak n(f)\}$ for all $\mathfrak n\in\max(\widehat{R})$. Let $\mathfrak n\in\max(\widehat{R})$.

\smallskip
\noindent
CASE 1: \, $\mathsf v_\mathfrak n(e)\not=\mathsf v_\mathfrak n(f)$.

Since $\mathsf v_\mathfrak n(e)\not=\mathsf v_\mathfrak n(\eta f)$ it follows that $\mathsf v_\mathfrak n(e+\eta f)=\min \{\mathsf v_\mathfrak n(e),\mathsf v_\mathfrak n(\eta f) \}=\min \{ \mathsf v_\mathfrak n(e),\mathsf v_\mathfrak n(f) \}$.

\smallskip
\noindent
CASE 2: \, $\mathsf v_\mathfrak n(e)=\mathsf v_\mathfrak n(f)$.

Since $\mathfrak n\in\mathfrak{I}$ we obtain that $a+\eta b\not\in\mathfrak n$. Consequently, $\mathsf v_\mathfrak n(e+\eta f)=\mathsf v_\mathfrak n(d(a+\eta b))=\mathsf v_\mathfrak n(d)+\mathsf v_\mathfrak n(a+\eta b)=\mathsf v_\mathfrak n(d)=\min \{ \mathsf v_\mathfrak n(e),\mathsf v_\mathfrak n(f) \}$.

Since $(\mathsf v_\mathfrak q(e+\eta f))_{\mathfrak q\in\max(\widehat{R})}\in V(\mathfrak m^{\bullet})$ we infer that $x=(\mathsf v_\mathfrak q(e+\eta f))_{\mathfrak q\in\max(\widehat{R})}=y$, a contradiction.

2. (a) $\Rightarrow$ (b) Let $\widehat{R}^{\times}/R^{\times}$ be finite and $R$ not a discrete valuation domain. It follows from \cite[Theorems 4.2 and 4.3]{Re13a} that $R$ is Noetherian and $\widehat{R}/(R:\widehat{R})$ is finite. Since $R$ is not a discrete valuation domain we have $(R:\widehat{R})\subset\mathfrak m$, and thus $|R/\mathfrak m|<\infty$.

(b) $\Rightarrow$ (c) This is clear, since $(R:\widehat{R})\not=\{0\}$.

(c) $\Rightarrow$ (a) The assertion is clear if $R$ is a discrete valuation domain. Now let $\widehat{R}$ be a finitely generated $R$-module and $|R/\mathfrak m|<\infty$. We have $\widehat{R}/\mathfrak q$ is a finite-dimensional $R/\mathfrak m$-vector space for all $\mathfrak q\in\max(\widehat{R})$, hence $\widehat{R}/\mathfrak q$ is finite for all $\mathfrak q\in\max(\widehat{R})$. Since $\widehat{R}$ is a principal ideal domain we infer that $\widehat{R}/I$ is finite for every nonzero ideal $I$ of $\widehat{R}$. Consequently, $\widehat{R}/(R:\widehat{R})$ is finite, and thus $\widehat{R}/R$ is finite. Since $R$ is local it follows from \cite[Proposition 3.5]{Re13a} that $\widehat{R}^{\times}/R^{\times}$ is finite.

(a) $\Leftrightarrow$ (d) This is an immediate consequence of \cite[Theorem 2.9.2.4]{Ge-HK06a}.

3. This follows from \cite[Theorem 2.7]{HK-Ha-Ka04}.
\end{proof}

\medskip
\begin{corollary}\label{5.7}
Let $(R,\mathfrak m)$ be a one-dimensional local domain with $(R:\widehat{R})\not=\{0\}$.
\begin{enumerate}
\item If $\widehat{R}$ is a semilocal principal ideal domain and $|\max(\widehat{R})|\leq |R/\mathfrak m|$, then $R^{\bullet}$ is ring-like.

\smallskip
\item If $R$ is Noetherian, then $R^{\bullet}$ is ring-like.

\smallskip
\item If $R$ is Noetherian and $|\max(\widehat{R})|\leq |R/\mathfrak m|<\infty$, then $R^{\bullet}$ is strongly ring-like.
\end{enumerate}
\end{corollary}

\begin{proof}
1. This follows from Proposition \ref{5.6}.

2. By Lemma \ref{5.2}, $R^{\bullet}$ is a finitely primary monoid. By Proposition \ref{5.6}.3 it is sufficient to show that $\widehat{R}^{\times}/R^{\times}$ is finite or $V(\mathfrak m^{\bullet})$ has a smallest element. Let $\widehat{R}^{\times}/R^{\times}$ be infinite. It follows from Proposition \ref{5.6}.2 that $R/\mathfrak m$ is infinite. Since $\widehat{R}$ is semilocal we have $|\max(\widehat{R})|\leq |R/\mathfrak m|$, and thus $V(\mathfrak m^{\bullet})$ has a smallest element by Proposition \ref{5.6}.1.

3. It is an immediate consequence of Proposition \ref{5.6} that $R^{\bullet}$ is strongly ring-like.
\end{proof}

The next example shows that Corollary \ref{5.7}.3 does not hold true for Mori domains.

\medskip
\begin{example}\label{5.8}
There are one-dimensional local seminormal Mori domains $R$ with $(R:\widehat{R})\not=\{0\}$ satisfying the following properties:
\begin{itemize}
\item $\widehat{R}$ is a discrete valuation domain,
\item $|\max(\widehat{R})|\leq |R/\mathfrak m|<\infty$ and $\widehat{R}^{\times}/R^{\times}$ is infinite,
\item $R$ is integrally closed or $\overline{R}=\widehat{R}$,
\item $R^{\bullet}$ is finitely primary but not strongly ring-like.
\end{itemize}
\end{example}

\begin{proof}
Let $K$ be a finite field and $L$ an infinite extension field of $K$. Let $X$ be an indeterminate over $L$. Set $R=K+XL[\![X]\!]$ and $\mathfrak m=XL[\![X]\!]$. Then $\mathfrak m$ is the unique nonzero prime ideal of $R$. In particular, $R$ is one-dimensional and local. Clearly, $\widehat{R}=L[\![X]\!]$ is a discrete valuation domain. Observe that $R$ is seminormal, hence $R$ is a Mori domain and $(R:\widehat{R})\not=\{0\}$. We obtain that $R^{\bullet}$ is finitely primary. Observe that $R/\mathfrak m\cong K$, and thus $|\max(\widehat{R})|=1\leq |R/\mathfrak m|=|K|<\infty$. Since $K$ is finite and $L$ is infinite, it follows that $\widehat{R}^{\times}/R^{\times}\cong L^{\times}/K^{\times}$ is infinite. Therefore, $R^{\bullet}$ is not strongly ring-like.

\smallskip
\noindent
CASE 1: \, $L/K$ is purely transcendental (e.g., $L=K(Y)$ for some indeterminate $Y$ over $K$). It is easy to prove that $R$ is integrally closed.

\smallskip
\noindent
CASE 2: \, $L/K$ is an algebraic field extension (e.g., $L$ is an algebraic closure of $K$). Observe that $\overline{R}=\widehat{R}$.
\end{proof}

\medskip
\begin{proposition}\label{5.9}
Let $R$ be a Mori domain such that $(R:\widehat{R})\not=\{0\}$ and $\mathfrak p\in\mathfrak{X}(R)$.
\begin{enumerate}
\item $R_\mathfrak p^{\bullet}$ is a finitely primary monoid.

\smallskip
\item For all $\mathfrak m\in {\rm spec}(\widehat{R})$ with $\mathfrak m\cap R=\mathfrak p$ it follows that $\mathfrak m\in\mathfrak{X}(\widehat{R})$. In particular, $\{\mathfrak m\in {\rm spec}(\widehat{R})\mid\mathfrak m\cap R=\mathfrak p\}=\{\mathfrak m\in\mathfrak{X}(\widehat{R})\mid\mathfrak m\cap R=\mathfrak p\}$.

\smallskip
\item If $\mathfrak p\supset (R:\widehat{R})$, then $R_\mathfrak p^{\bullet}$ is strongly ring-like of rank at most two if and only if $R_\mathfrak p$ is Noetherian, $|R/\mathfrak p|<\infty$ and $|\{\mathfrak m\in\mathfrak{X}(\widehat{R})\mid\mathfrak m\cap R=\mathfrak p\}|\leq 2$.
\end{enumerate}
\end{proposition}

\begin{proof}
Clearly, $R_\mathfrak p$ is a one-dimensional local Mori domain, $\widehat{R_\mathfrak p}=\widehat{R}_\mathfrak p$ (since $\widehat{R}$ is a Krull domain) and $\{0\}\not=(R:\widehat{R})\subset (R_\mathfrak p:\widehat{R_\mathfrak p})$. If $\mathfrak{m}\in {\rm spec}(\widehat{R})$, then set $\mathfrak{m}_{\mathfrak{p}}=(R\setminus\mathfrak{p})^{-1}\mathfrak{m}$.

\smallskip
1. It follows from Lemma \ref{5.2}.1 that $R_\mathfrak p^{\bullet}$ is finitely primary.

\smallskip
2. Let $\mathfrak m\in {\rm spec}(\widehat{R})$ with $\mathfrak m\cap R=\mathfrak p$. Since $\widehat{R}$ is a Krull domain, there is some $\mathfrak q\in\mathfrak{X}(\widehat{R})$ such that $\mathfrak q\subset\mathfrak m$. This implies that $\{0\}\not=\mathfrak q_\mathfrak p\subset\mathfrak m_\mathfrak p$. Since $\widehat{R}_\mathfrak p$ is a semilocal principal ideal domain it follows that $\mathfrak q_\mathfrak p=\mathfrak m_\mathfrak p$, and thus $\mathfrak q=\mathfrak m$. We infer that $\mathfrak m\in\mathfrak{X}(\widehat{R})$.

3. Let $\mathfrak p\supset (R:\widehat{R})$. Note that $\widehat{R_\mathfrak p}$ is a semilocal principal ideal domain. We infer by 2 that $\mathfrak{X}(\widehat{R_\mathfrak p})={\rm spec}(\widehat{R_\mathfrak p})\setminus\{(0)\}={\rm spec}(\widehat{R}_\mathfrak p)\setminus\{(0)\}=\{\mathfrak m_\mathfrak p\mid\mathfrak m\in {\rm spec}(\widehat{R})\setminus\{(0)\},\mathfrak m\cap R\subset\mathfrak p\}=\{\mathfrak m_\mathfrak p\mid\mathfrak m\in {\rm spec}(\widehat{R}),\mathfrak m\cap R=\mathfrak p\}=\{\mathfrak m_\mathfrak p\mid\mathfrak m\in\mathfrak{X}(\widehat{R}),\mathfrak m\cap R=\mathfrak p\}$. Note that $\varphi:\{\mathfrak m\in\mathfrak{X}(\widehat{R})\mid\mathfrak m\cap R=\mathfrak p\}\rightarrow\{\mathfrak m_\mathfrak p\mid\mathfrak m\in\mathfrak{X}(\widehat{R}),\mathfrak m\cap R=\mathfrak p\}$ defined by $\varphi(\mathfrak{m})=\mathfrak m_\mathfrak p$ for all $\mathfrak m\in\mathfrak{X}(\widehat{R})$ with $\mathfrak m\cap R=\mathfrak p$ is a bijection. Consequently, $|\mathfrak{X}(\widehat{R_\mathfrak p})|=|\{\mathfrak m\in\mathfrak{X}(\widehat{R})\mid\mathfrak m\cap R=\mathfrak p\}|$.

First suppose that $R_\mathfrak p^{\bullet}$ is strongly ring-like of rank at most two. Then $\widehat{R}_\mathfrak p^{\times}/R_\mathfrak p^{\times}$ is finite. It follows by \cite[Theorem 2.6.5.3]{Ge-HK06a} that $R_\mathfrak p$ is not a discrete valuation domain. Therefore, we obtain by Proposition \ref{5.6} that $R_\mathfrak p$ is Noetherian and $|R_\mathfrak p/\mathfrak p_\mathfrak p|<\infty$. Since $f:R/\mathfrak p\rightarrow R_\mathfrak p/\mathfrak p_\mathfrak p$ defined by $f(x+\mathfrak p)=x+\mathfrak p_\mathfrak p$ for all $x\in R$ is a ring monomorphism, we have $|R/\mathfrak p|<\infty$. We infer that $|\{\mathfrak m\in\mathfrak{X}(\widehat{R})\mid\mathfrak m\cap R=\mathfrak p\}|=|\mathfrak{X}(\widehat{R_\mathfrak p})|=|\mathfrak{X}(\widehat{R_\mathfrak p}^{\bullet})|\leq 2$.

Conversely suppose that $R_\mathfrak p$ is Noetherian, $|R/\mathfrak p|<\infty$ and $|\{\mathfrak m\in\mathfrak{X}(\widehat{R})\mid\mathfrak m\cap R=\mathfrak p\}|\leq 2$. Then $\mathfrak p$ is a maximal ideal of $R$ whence $R/\mathfrak p \cong R_{\mathfrak p}/\mathfrak p_{\mathfrak p}$. Thus we obtain that have
\[
|\max(\widehat{R_\mathfrak p})|=|\mathfrak{X}(\widehat{R_\mathfrak p}^{\bullet})|=|\mathfrak{X}(\widehat{R_\mathfrak p})|=|\{\mathfrak m\in\mathfrak{X}(\widehat{R})\mid\mathfrak m\cap R=\mathfrak p\}|\leq 2 \le |R_{\mathfrak p}/\mathfrak p_{\mathfrak p}| = |R/\mathfrak p|<\infty \,.
\]
Therefore, it follows from Corollary \ref{5.7} that $R_\mathfrak p^{\bullet}$ is strongly ring-like. Clearly, $R_\mathfrak p^{\bullet}$ is of rank at most two.
\end{proof}

Our next goal is to show that for strongly ring-like monoids of rank at most the weak successive distance $\delta_w (\cdot)$ is finite (see Proposition \ref{5.12}).

\medskip
\begin{lemma}\label{x.vo2}
Let $H$ be a ring-like monoid of rank $s$ and exponent $\alpha$ and $\{p_i\mid i\in [1,s]\}$ a system of representatives of the prime elements of $\widehat H$ such that $\alpha$ and $\{p_i\mid i\in [1,s]\}$ are suitably chosen and such that either $s\geq 2$ or $\widehat{H}^{\times}/H^{\times}$ is finite.
\begin{enumerate}
\item If $u\in H$, $i\in [1,s]$ and $\mathsf v_{p_i}(u)\geq 2\alpha$, then $u\in\mathcal{A}(H)$ if and only if $p_iu\in\mathcal{A}(H)$.

\smallskip
\item There are some $M_1,C_1\in\mathbb{N}$ such that for each $a\in H$ and all adjacent $k, \ell\in\mathsf L(a)$ such that $\max\{k, \ell\}+M_1\leq\max\mathsf L(a)$ it follows that ${\rm Dist}(\mathsf Z_k(a),\mathsf Z_{\ell}(a))\leq C_1$.
\end{enumerate}
\end{lemma}

\begin{proof}
1. This follows from \cite[Proposition 4.4]{Ha09c}.

\smallskip
2. We distinguish two cases.

\smallskip
\noindent
CASE 1: \, $s=1$.

Then $H_{\red}$ is finitely generated by Remark \ref{5.3}.2.
 Thus Theorem \ref{3.1} implies that $\delta (H) < \infty$ whence the assertion follows.

\smallskip
\noindent
CASE 2: \, $s\geq 2$.

We infer by \cite[Theorem 3.1.5]{Ge-HK06a} that $\min \mathsf L(a)\leq 2\alpha$ for all $a\in H$. It follows by \cite[Theorem 4.14]{Ha09c} that there are some $M_1,C_2\in\mathbb{N}$ such that for all $a\in H$ and all adjacent $k, \ell\in\mathsf L(a)$ for which $\min\mathsf L(a)+M_1\leq\min\{k,\ell\}\leq\max\{k,\ell\}\leq\max\mathsf L(a)-M_1$ it follows that ${\rm Dist}(\mathsf Z_k(a),\mathsf Z_{\ell}(a))\leq C_2$. By Lemma \ref{5.1}, $\Delta(H)$ is finite and nonempty. Set $C_1=\max\{C_2,M_1+2\alpha+\max \Delta(H) \}$. Let $a\in H$ and $k, \ell\in\mathsf L(a)$ be adjacent such that $\max\{k,\ell\}+M_1\leq\max\mathsf L(a)$.

CASE 2.1: \, $\min\{k,\ell\}\geq M_1+2\alpha$. Obviously, $\min\mathsf L(a)+M_1\leq\min\{k,\ell\}$, and thus ${\rm Dist}(\mathsf Z_k(a),\mathsf Z_{\ell}(a))\leq C_2\leq C_1$.

CASE 2.2: \, $\min\{k,\ell\}<M_1+2\alpha$. It is easy to see that ${\rm Dist}(\mathsf Z_k(a),\mathsf Z_{\ell} (a))\leq\max\{k,\ell\}\leq\min\{k,\ell\}+\max \Delta(H) \leq M_1+2\alpha+\max \Delta(H) \leq C_1$.
\end{proof}

\medskip
\begin{lemma}\label{x.vo3}
Let $H$ be a strongly ring-like monoid of rank $2$ and exponent $\alpha$, $\{p_1,p_2\}$ a system of representatives of the prime elements of $\widehat{H}$ such that $\alpha$ and $\{p_1,p_2\}$ are suitably chosen and $(\mu_1,\mu_2)$ the smallest element of $\{(\mathsf v_{p_1}(a),\mathsf v_{p_2}(a))\mid a\in H\setminus H^{\times}\}$. Set $\mathbb{A}(H)=\{q\in\mathcal{A}(H)\mid\mathsf v_{p_1}(q)\leq 2\alpha,\mathsf v_{p_2}(q)\leq 2\alpha\}$.
\begin{enumerate}
\item For every $a\in H$ we have $\min\{\mathsf v_{p_i}(a)-\mu_i\max\mathsf L(a)\mid i\in\{1,2\}\}<\alpha$.

\smallskip
\item There are $L,E\in\mathbb{N}$ such that for all $u\in\mathbb{A}(H)$, $r\in\mathbb{N}$ and $(u_i)_{i=1}^r\in\mathcal{A}(H)^r$ with $\mathsf v_{p_1}(u)=\mu_1$, $|\{i\in [1,r]\mid\mathsf v_{p_1}(u_i)=\mu_1\}|>L$ and $\max\{\mathsf v_{p_2}(u_i)\mid i\in [1,r]\}>E$ it follows that $u\mid_H\prod_{i=1}^r u_i$ and $r-1\in\mathsf L(u^{-1}\prod_{i=1}^r u_i)$.

\smallskip
\item For every $N\in\mathbb{N}$ there is some $D\in\mathbb{N}$ such that for all $a\in H$ and $k,\ell\in\mathsf L(a)$ for which $\min\{k,\ell\}+N\geq\max\mathsf L(a)$ it follows that $\mathsf d(\mathsf Z_k(a),\mathsf Z_\ell(a))\leq D|\ell-k|$.
\end{enumerate}
\end{lemma}

\begin{proof}
1. This follows from \cite[Lemma 5.1]{Ha09c}.

\smallskip
2. It follows from \cite[Lemma 5.2]{Ha09c} that there are some $L,E\in\mathbb{N}$ such that for all $u\in\mathbb{A}(H)$, $(u_i)_{i=1}^L\in\mathcal{A}(H)^L$ with $\mathsf v_{p_1}(u)=\mu_1$, $\mathsf v_{p_1}(u_i)=\mu_1$ for all $i\in [1,L]$ and $v\in\mathcal{A}(H)$ with $\mathsf v_{p_2}(v)\geq E$ it follows that $u\prod_{i=1}^L w_i=v\prod_{i=1}^L u_i$ for some $(w_i)_{i=1}^L\in\mathcal{A}(H)^L$. Now let $u\in\mathbb{A}(H)$, $r\in\mathbb{N}$ and $(u_i)_{i=1}^r\in\mathcal{A}(H)^r$ be such that $\mathsf v_{p_1}(u)=\mu_1$, $|\{i\in [1,r]\mid\mathsf v_{p_1}(u_i)=\mu_1\}|>L$ and $\max\{\mathsf v_{p_2}(u_i)\mid i\in [1,r]\}>E$. There are some $j\in [1,r]$ such that $\mathsf v_{p_2}(u_j)=\max\{\mathsf v_{p_2}(u_i)\mid i\in [1,r]\}$ and $I\subset [1,r]\setminus\{j\}$ such that $|I|=L$ and $\mathsf v_{p_1}(u_i)=\mu_1$ for all $i\in I$. Consequently, there is some $(w_i)_{i=1}^L\in\mathcal{A}(H)^L$ such that $u\prod_{i=1}^L w_i=u_j\prod_{i\in I} u_i$. This implies that $\prod_{i=1}^r u_i=u\prod_{i=1}^L w_i\prod_{i\in [1,r]\setminus (I\cup\{j\})} u_i$, and thus $u\mid_H\prod_{i=1}^r u_i$ and $r-1=L+r-(L+1)\in\mathsf L(u^{-1}\prod_{i=1}^r u_i)$.

\smallskip
3. Let $N\in\mathbb{N}$. Without restriction let $H$ be reduced. By 2 there are some $L_1,E_1\in\mathbb{N}$ such that for all $u\in\mathbb{A}(H)$, $r\in\mathbb{N}$ and $(u_i)_{i=1}^r\in\mathcal{A}(H)^r$ with $\mathsf v_{p_1}(u)=\mu_1$, $|\{i\in [1,r]\mid\mathsf v_{p_1}(u_i)=\mu_1\}|>L_1$ and $\max\{\mathsf v_{p_2}(u_i)\mid i\in [1,r]\}>E_1$ it follows that $u\mid_H\prod_{i=1}^r u_i$ and $r-1\in\mathsf L(u^{-1}\prod_{i=1}^r u_i)$.

Analogously, it follows by 2 that there are some $L_2,E_2\in\mathbb{N}$ such that for all $u\in\mathbb{A}(H)$, $r\in\mathbb{N}$ and $(u_i)_{i=1}^r\in\mathcal{A}(H)^r$ with $\mathsf v_{p_2}(u)=\mu_2$, $|\{i\in [1,r]\mid\mathsf v_{p_2}(u_i)=\mu_2\}|>L_2$ and $\max\{\mathsf v_{p_1}(u_i)\mid i\in [1,r]\}>E_2$ it follows that $u\mid_H\prod_{i=1}^r u_i$ and $r-1\in\mathsf L(u^{-1}\prod_{i=1}^r u_i)$.

Set $E=\max\{E_1,E_2,\mu_1(N+1)+\alpha,\mu_2(N+1)+\alpha\}$ and $\mathbb{B}=\{q\in\mathcal{A}(H)\mid\mathsf v_{p_1}(q)\leq E,\mathsf v_{p_2}(q)\leq E\}$. Let $H(\mathbb{B})$ be the submonoid of $H$ generated by $\mathbb{B}$. Note that $\mathcal{A}(H(\mathbb{B}))=\mathbb{B}$. Since $\mathbb{B}$ is finite we have $H(\mathbb{B})$ is (quasi) finitely generated. By Theorem \ref{3.1} and Lemma \ref{3.5}, there is some $C\in\mathbb{N}$ such that for all $a\in H(\mathbb{B})$ and all $k, \ell\in\mathsf L_{H(\mathbb{B})}(a)$ it follows that $\mathsf d(\mathsf Z_{H(\mathbb{B}),k}(a),\mathsf Z_{H(\mathbb{B}), \ell}(a))\leq C|\ell-k|$. Set $D=\max\{C,N(\mu_1+1)+L_1+\alpha,N(\mu_2+1)+L_2+\alpha\}$. It suffices to show by induction that for all $r\in\mathbb{N}_0$ and $a\in H$ such that $\max\mathsf L(a)=r$ we have for all $k, \ell\in\mathsf L(a)$ for which $\min\{k, \ell\}+N\geq\max\mathsf L(a)$ it follows that $\mathsf d(\mathsf Z_k(a),\mathsf Z_\ell(a))\leq D|\ell-k|$. Let $r\in\mathbb{N}_0$, $a\in H$ and $k, \ell\in\mathsf L(a)$ be such that $\max\mathsf L(a)=r$ and $\min\{k,\ell\}+N\geq\max\mathsf L(a)$. We have to show that $\mathsf d(\mathsf Z_k(a),\mathsf Z_{\ell}(a))\leq D|\ell - k|$. Without restriction let $k<\ell$. By 1 we can assume without restriction that $\mathsf v_{p_1}(a)-\mu_1\max\mathsf L(a)<\alpha$. If $\max\mathsf L(a)\leq D$, then $\mathsf d(\mathsf Z_k(a),\mathsf Z_{\ell}(a))\leq\max\{k,\ell\}\leq\max\mathsf L(a)\leq D\leq D|\ell - k|$. Now let $\max\mathsf L(a)>D$.

Next we show that for all $s\in\mathbb{N}_0$ and $(w_i)_{i=1}^s\in\mathcal{A}(H)^s$ with $a=\prod_{i=1}^s w_i$ and $s+N\geq\max\mathsf L(a)$ it follows that $|\{i\in [1,s]\mid\mathsf v_{p_1}(w_i)=\mu_1\}|>L_1$ and $\mathsf v_{p_1}(w_j)\leq\mu_1(N+1)+\alpha$ for all $j\in [1,s]$. Let $s\in\mathbb{N}_0$ and $(w_i)_{i=1}^s\in\mathcal{A}(H)^s$ be such that $a=\prod_{i=1}^s w_i$ and $s+N\geq\max\mathsf L(a)$. We have $|\{i\in [1,s]\mid\mathsf v_{p_1}(w_i)=\mu_1\}|=s-|\{i\in [1,s]\mid\mathsf v_{p_1}(w_i)\not=\mu_1\}|\geq s-\sum_{i=1}^s (\mathsf v_{p_1}(w_i)-\mu_1)=s-(\mathsf v_{p_1}(a)-s\mu_1)=s(\mu_1+1)-\mathsf v_{p_1}(a)>s(\mu_1+1)-\alpha-\mu_1\max\mathsf L(a)\geq (\max\mathsf L(a)-N)(\mu_1+1)-\alpha-\mu_1\max\mathsf L(a)=\max\mathsf L(a)-N(\mu_1+1)-\alpha>D-N(\mu_1+1)-\alpha\geq L_1$. Let $j\in [1,s]$. Then $\mathsf v_{p_1}(w_j)=\mathsf v_{p_1}(a)-\sum_{i=1,i\not=j}^s\mathsf v_{p_1}(w_i)\leq\mathsf v_{p_1}(a)-(s-1)\mu_1\leq\alpha+\mu_1\max\mathsf L(a)-(s-1)\mu_1=\mu_1(\max\mathsf L(a)-s+1)+\alpha\leq\mu_1(N+1)+\alpha$.

\smallskip
We continue with the following two assertions.
\begin{enumerate}
\item[{\bf A1.}\,] There are some $u_1\in\mathbb{A}(H)$ and $(u_i)_{i=2}^k\in\mathcal{A}(H)^{k-1}$ such that $\mathsf v_{p_1}(u_1)=\mu_1$ and $a=\prod_{i=1}^k u_i$.

\item[{\bf A2.}\,] There are some $v_1\in\mathbb{A}(H)$ and $(v_j)_{j=2}^{\ell}\in\mathcal{A}(H)^{\ell-1}$ such that $\mathsf v_{p_1}(v_1)=\mu_1$ and $a=\prod_{j=1}^{\ell} v_j$.
\end{enumerate}

{\it Proof of} \ {\bf A1.} and {\bf A2.} By symmetry it is sufficient to prove {\bf A1.}
Since $k+N\geq\max\mathsf L(a)$ there is some $(w_i)_{i=1}^k\in\mathcal{A}(H)^k$ such that $\mathsf v_{p_1}(w_1)=\mathsf v_{p_1}(w_2)=\mu_1$ and $a=\prod_{i=1}^k w_i$. It is obvious that $\mu_1\leq 2\alpha$. If $\mathsf v_{p_2}(w_1)\leq 2\alpha$ or $\mathsf v_{p_2}(w_2)\leq 2\alpha$, then $w_1\in\mathbb{A}(H)$ or $w_2\in\mathbb{A}(H)$ and we are done. Now suppose that $\mathsf v_{p_2}(w_1)>2\alpha$ and $\mathsf v_{p_2}(w_2)>2\alpha$. We have $w_1=\beta p_1^{\mu_1}p_2^t$ and $w_2=\gamma p_1^{\mu_1}p_2^s$ for some $\beta,\gamma\in\widehat{H}^{\times}$ with $t=\mathsf v_{p_2}(w_1)$ and $s=\mathsf v_{p_2}(w_2)$. Set $u_1=\beta p_1^{\mu_1}p_2^{2\alpha}$, $u_2=\gamma p_1^{\mu_1}p_2^{t+s-2\alpha}$ and $u_i=w_i$ for all $i\in [3,k]$. We infer by Lemma \ref{x.vo2}.1 that $u_1,u_2\in\mathcal{A}(H)$. It is clear that $u_1\in\mathbb{A}(H)$ and $a=\prod_{i=1}^k u_i$. \qed ({\bf A1} and {\bf A2})

\smallskip
\noindent
CASE 1: \, $\max(\{\mathsf v_{p_2}(u_i)\mid i\in [1,k]\}\cup\{\mathsf v_{p_2}(v_i)\mid i\in [1,\ell]\})\leq E_1$.

 Observe that $\{u_i\mid i\in [1,k]\}\cup\{v_j\mid j\in [1,\ell]\}\subset\mathbb{B}$. Therefore, we have $k,\ell\in\mathsf L_{H(\mathbb{B})}(a)$, and thus $\mathsf d(\mathsf Z_k(a),\mathsf Z_{\ell}(a))\leq\mathsf d(\mathsf Z_{H(\mathbb{B}),k}(a),\mathsf Z_{H(\mathbb{B}), \ell}(a))\leq C|\ell - k|\leq D|\ell - k|$.

\smallskip
\noindent
CASE 2: \, $\max\{\mathsf v_{p_2}(u_i)\mid i\in [1,k]\}>E_1$.

We have $k-1,\ell-1\in\mathsf L(v_1^{-1}a)$, $\max\mathsf L(v_1^{-1}a)<\max\mathsf L(a)$ and $\min\{k-1,\ell-1\}+N\geq\max\mathsf L(v_1^{-1}a)$. Therefore, it follows by the induction hypothesis that $\mathsf d(\mathsf Z_{k-1}(v_1^{-1}a),\mathsf Z_{\ell-1}(v_1^{-1}a))\leq D|(k-1)-(\ell-1)|=D|\ell - k|$. There are some $x\in\mathsf Z_{k-1}(v_1^{-1}a)$ and $y\in\mathsf Z_{\ell-1}(v_1^{-1}a)$ such that $\mathsf d(\mathsf Z_{k-1}(v_1^{-1}a),\mathsf Z_{\ell-1}(v_1^{-1}a))=\mathsf d(x,y)$. Clearly, $v_1x\in\mathsf Z_k(a)$ and $v_1y\in\mathsf Z_{\ell}(a)$. Consequently, $\mathsf d(\mathsf Z_k(a),\mathsf Z_{\ell}(a))\leq\mathsf d(v_1x,v_1y)=\mathsf d(x,y)\leq D|\ell - k|$.

\smallskip
\noindent
CASE 3: \, $\max\{\mathsf v_{p_2}(v_i)\mid i\in [1,\ell]\}>E_1$.

This follows by analogy with CASE 2.
\end{proof}

\medskip
\begin{proposition}\label{5.12}
Let $H$ be a strongly ring-like monoid of rank $s\le 2$. Then $\mathsf c_{\mon}(H)<\infty$ and $\delta_w(H)<\infty$.
\end{proposition}

\begin{proof}
If $s=1$, then the assertion follows from Remark \ref{5.3}.2. Now suppose that $s=2$.
We infer by \cite[Theorem 5.3]{Ha09c} that $\mathsf c_{\mon}(H)<\infty$. Both Lemmas, \ref{x.vo2}.2 and \ref{x.vo3}.3, together with Lemma \ref{5.1} show that the assumptions of Lemma \ref{3.6} are satisfied whence we obtain that $\delta_w ( H) < \infty$.
\end{proof}

\smallskip
Now we formulate the main result of the present section. The crucial point in its proof is that the finiteness of the weak successive distance (achieved in Proposition \ref{5.12}) and the validity of the Structure Theorem for Sets of Lengths preserve the finiteness of the monotone catenary degree when passing to direct products (see Theorem \ref{3.8}).

\medskip
\begin{theorem}\label{5.13}
Let $H$ be a $v$-Noetherian weakly Krull monoid with $\emptyset\ne\mathfrak f = (H\DP\widehat H)$ such that $H_{\mathfrak p}$ is strongly ring-like of rank $s_{\mathfrak p}\le 2$ for each $\mathfrak{p}\in\mathfrak{X}(H)$ with $\mathfrak p\supset\mathfrak f$.
\begin{enumerate}
\item $\mathcal I_v^*(H)$ is a \C-monoid and if $\mathcal C_v (H)$ is finite, then $H$ is a \C-monoid.

\smallskip
\item $\mathcal I_v^*(H)$ satisfies the Structure Theorem for Sets of Lengths and the Structure Theorem for Unions.

\smallskip
\item $\delta_w\big( \mathcal I_v^*(H)\big)<\infty$, and $\mathsf c_{\mon}\big(\mathcal I_v^*(H)\big)<\infty$.
\end{enumerate}
\end{theorem}

\begin{proof}
By \cite[Theorem 5.5]{Ge-Ka-Re15a}, there is a monoid isomorphism
\begin{equation} \label{structure}
\chi\colon\mathcal I_v^*(H)\to\mathcal F (\mathcal P)\times\prod_{\mathfrak p\in\mathcal P^*} (H_\mathfrak p)_\red\quad\text{satisfying}\quad\chi\t\mathcal P =\id_{\mathcal P}.
\end{equation}
where $\mathcal P =\{\mathfrak p\in\mathfrak X(R)\mid\mathfrak
p\not\supset\mathfrak f\}$ and $\mathcal P^*=\mathfrak X (H)\setminus\mathcal P$. Observe that $\mathcal P^*$ is finite.

\smallskip
1. Clearly, $\mathcal F (\mathcal P)$ is a C-monoid. It follows from \cite[Corollary 2.9.8]{Ge-HK06a} that all $H_{\mathfrak p}$ (and thus all $(H_\mathfrak p)_\red$) with $\mathfrak p\in\mathcal P^*$ are C-monoids. Since $\widehat{H_{\mathfrak p}}^{\times}/H_\mathfrak p^{\times}$ is finite for all $\mathfrak p\in\mathcal P^*$ we infer by \cite[Theorem 2.9.16]{Ge-HK06a} that $\mathcal F (\mathcal P)\times\prod_{\mathfrak p\in\mathcal P^*} (H_\mathfrak p)_\red\cong(\mathcal F (\mathcal P)\times\prod_{\mathfrak p\in\mathcal P^*} H_\mathfrak p)_\red$ is a C-monoid. Therefore, $\mathcal I_v^*(H)$ is a C-monoid. Now let $\mathcal C_v (H)$ be finite. We have $H_{\red}\cong\{aH\mid a\in H\}$, $\{aH\mid a\in H\}\subset\mathcal I_v^*(H)$ is a saturated submonoid and $\mathcal I_v^*(H)/\{aH\mid a\in H\}\subset\mathcal{C}_v(H)$ is finite. Therefore, it follows by \cite[Theorems 2.9.10 and 2.9.16]{Ge-HK06a} that $H$ is a C-monoid.

\smallskip
2. and 3. Note that all $H_{\mathfrak p}$ (and thus all $(H_\mathfrak p)_\red$) with $\mathfrak p\in\mathcal P^*$ satisfy the Structure Theorem for Sets of Lengths by Lemma \ref{5.1}, and they have finite weak successive distance and finite equal catenary degree by Proposition \ref{5.12}. By \eqref{structure} and Theorem \ref{3.8}, $\mathcal I_v^*(H)$ satisfies the Structure Theorem for Sets of Lengths, $\delta_w \big(\mathcal I_v^*(H)\big)<\infty$, and $\mathsf c_{\eq}\big(\mathcal I_v^*(H)\big)<\infty$. We infer by Lemma \ref{3.5} that $\mathsf c_{\mon} \big( \mathcal I_v^*(H)\big)<\infty$. Since $\mathcal I_v^* (H)$ is a C-monoid by 1, and C-monoids satisfy the Structure Theorem for Unions by \cite[Theorems 3.10 and 4.2]{Ga-Ge09b}, $\mathcal I_v^*(H)$ satisfies the Structure Theorem for Unions.
\end{proof}

\medskip
\begin{corollary}\label{5.14}
Let $R$ be a weakly Krull Mori domain with $(R:\widehat{R})\not=\{0\}$ such that $R_\mathfrak p$ is Noetherian, $|R/\mathfrak p|<\infty$, and $|\{\mathfrak m\in\mathfrak{X}(\widehat{R})\mid\mathfrak m\cap R=\mathfrak p\}|\leq 2$ for each $\mathfrak p\in\mathfrak{X}(R)$ with $\mathfrak p\supset (R:\widehat{R})$. Then the monoid $\mathcal I_v^*(R)$ has finite weak successive distance and finite monotone catenary degree. \\
In particular, if $R$ is an order in a quadratic number field, then $\mathcal I^*(R)$ has finite weak successive distance and finite monotone catenary degree.
\end{corollary}

\noindent
{\it Comment.} The example given in Remark \ref{5.3}.4 shows that the assumption, that $|\{\mathfrak m\in\mathfrak{X}(\widehat{R})\mid\mathfrak m\cap R=\mathfrak p\}|\leq 2$ for each $\mathfrak p\in\mathfrak{X}(R)$, is crucial.

\begin{proof}
This follows from Theorem \ref{5.13} and Proposition \ref{5.9}.3. Suppose that $R$ is an order in a quadratic number field. Then $R$ is one-dimensional Noetherian whence $\mathcal I^*(R)=\mathcal I_v^*(R)$, and $\mathcal I_v^* (H)$ has finite weak successive distance and finite monotone catenary degree by the main statement.
\end{proof}


\providecommand{\bysame}{\leavevmode\hbox to3em{\hrulefill}\thinspace}
\providecommand{\MR}{\relax\ifhmode\unskip\space\fi MR }
\providecommand{\MRhref}[2]{%
  \href{http://www.ams.org/mathscinet-getitem?mr=#1}{#2}
}
\providecommand{\href}[2]{#2}

\end{document}